\documentclass[11pt]{amsart}
\usepackage{amsmath,amsfonts,latexsym,graphicx,amssymb,url}
\usepackage{hyperref,pdfsync}
\usepackage{amsmath,txfonts,pifont,bbding,pxfonts}
\usepackage{manfnt}
\usepackage[active]{srcltx}
\usepackage{cases}
\usepackage{wasysym,pstricks, enumerate,subfigure}
\usepackage{lscape,color}
\usepackage{soul}
\usepackage[none]{hyphenat}
\usepackage[mathscr]{euscript}
\usepackage{wrapfig, subfigure,graphicx}
\usepackage{enumitem}

\usepackage{caption}
\usepackage{cleveref}

\newcommand{\diam}{\text{{\rm diam}}}

\newcommand{\R}{{\mathbb R}} 

\renewcommand{\(}{\left(}
\renewcommand{\)}{\right)}
\newtheorem{theorem}{Theorem}[section]

\newtheorem{lemma}[theorem]{Lemma}
\newtheorem{proposition}[theorem]{Proposition}
\newtheorem{definition}[theorem]{Definition}
\newtheorem{remark}[theorem]{Remark}

\begin{document}

\title[The Multi-marginal Monge Problem and an application to metasurfaces \today]{The Multi-marginal Monge Problem\\ and an application to metasurfaces}
\author[I. Altiner and C. E. Guti\'errez]{Irem Altiner and Cristian E. Guti\'errez}
\thanks{This research was partially supported by NSF Grants DMS-1600578 and DMS-2203555. \today}
%
%

\begin{abstract}
This paper studies the multi-marginal Monge problem in the setting of compact metric spaces proving existence and uniqueness of solutions when the cost function is Lipschitz. 
We apply the results obtained to solve an optics problem involving metalenses, that is, we design a refracting-reflecting metalens that preserves given energy distributions.
 
\end{abstract}

\maketitle

\tableofcontents

\setcounter{equation}{0}
\section{Introduction}

This paper considers the multimarginal Monge problem, described as follows.
Let $(D_j,d_j)$ be compact metric spaces for $1\leq j\leq N$, and let $c:D_1\times D_2\times \cdots \times D_N\to \R$ be a cost function.
Assume we are given Borel measures $\mu_j$ on $D_j$ satisfying the mass conservation condition $\mu_1(D_1)=\mu_j(D_j)$ for all $1\leq j \leq N$.

A multivalued map $T_j:D_1\to D_j$ is said to be measure-preserving with respect to $(\mu_1,\mu_j)$ if for every Borel set $E\subset D_j$, the set $T_j^{-1}(E)$ is $\mu_1$-measurable and
\[
\mu_1\big(T_j^{-1}(E)\big) = \mu_j(E).
\]
The multimarginal, or multivariate, Monge problem seeks to find maps $T_j:D_1\to D_j$, with $2\leq j\leq N$, that are measure-preserving with respect to $(\mu_1,\mu_j)$, such that $T_j(x_1)$ is single-valued for $\mu_1$-almost every $x_1 \in D_1$, and minimize the total cost
\[
\int_{D_1} c\big(x_1,T_2(x_1),\cdots ,T_N(x_1)\big)\,d\mu_1(x_1)
\]
among all such families of measure-preserving maps.

This is a problem that generated recently great interest and has been considered in various settings, see \cite{2015-pass:multimarginalOTtheoryandapplications}, 
\cite{2023-Pass-Jimenez:mongesolutionsuniquenessmultimarginalgraphtheory}, 
\cite{2022-PassandJimenez:Mongesolutionsuniquenessmultimarginal}, 
and \cite{2022-vargas-jimenez:phdthesismultimarginal}; see also \cite[Section 2.1] {RR:masstransport}.

The main goal of this paper is to show that if the cost function $c$ is Lipschitz (inequality~\eqref{eq:cost is Lipschitz}) and the $c$-normal mapping (or $c$-superdifferential, see Definition~\ref{def:c normal mapping}) is single-valued almost everywhere (see condition~\nameref{hyp:SV}), then the $c$-normal mapping solves the multimarginal Monge problem (Theorem~\ref{thm:existence of solutions to Monge problem}). Moreover, the solution is unique almost everywhere when $\mu_1$ is strictly positive (Theorem~\ref{thm:uniqueness of the  multimarginal problem}).
The technique used in this paper is an extension of the duality arguments used in \cite[Chapter 6]{2023-gutierrez:bookonOTandoptics} to deal with the Monge problem.

For domains $D_j \subset \R^n$ and $C^1$ cost functions $c$, the single-valuedness condition~\nameref{hyp:SV} holds when the map
\[
(x_2,\cdots,x_N)\mapsto \nabla_{x_1}c(x_1,x_2,\cdots,x_N)
\]
is injective for each $x_1\in D_1$ (Proposition~\ref{prop:twist implies SV}). In addition, Theorem~\ref{thm:representation formula for multimarginal} provides a representation formula for the gradient of the cost function (equation~\eqref{eq:representation formula for multimarginal}).

Our motivation for studying the multimarginal Monge problem arises from an application in optics related to metasurfaces. Indeed, in Section~\ref{subsec:optics application of multivariate monge}, using this theory we design a reflecting-refracting metalens obtaining Theorem \ref{thm:main result representation of the phase}. 
In this context, the cost function is given by~\eqref{multi-marginal cost}, and in Section~\ref{analysisofmultimarginalcost}, we verify that the map $(x_2,x_3)\mapsto \nabla_{x_1}c(x_1,x_2,x_3)$ is injective for all $x_1$ under given conditions on the metasurface.

\setcounter{equation}{0}
\section{$c$-Transforms}
Let $(D_j,d_j)$ be compact metric spaces with $1\leq j\leq N$, and let $c:D_1\times D_2\times \cdots \times D_N\to \R$ be a Lipschitz function, i.e., there exists a constant $K>0$ such that 
\begin{equation}\label{eq:cost is Lipschitz}
|c(x_1,\cdots ,x_N)-c(y_1,\cdots ,y_N)|\leq K\,\(d_1(x_1,y_1)+\cdots +d_N(x_N,y_N)\),
\end{equation}
for all $x_i,y_i\in D_i$, $1\leq i\leq N$.
\begin{definition}
Let $f_j\in C(D_j)$ with $1\leq j\leq N$. The $c$-transform of the vector $F=\(f_1,\cdots ,f_N\)$ is the $n$-vector $F^c=\(f^c_1,\cdots ,f^c_N\)$ having components 
\[
f_j^c(x_j)=\inf_{z_k\in D_k,k\neq j}
\left\{
c\(z_1,\cdots ,z_{j-1},x_j,z_{j+1},\cdots ,z_N\)
-
\sum_{k\neq j} f_k(z_k)
\right\}.
\]
The function $f_j^c$ is called the $j$-th $c$-transform of $F$.
\end{definition}
Since the spaces are compact and the functions are continuous, the $c$-transforms are clearly finite.
Notice that the definition of $c$-transform depends on the components of the vector $F=(f_1,\cdots ,f_{j-1},f_j,f_{j+1},\cdots ,f_N)$. For example,
if we modify one coordinate of this vector taking $G=(f_1,\cdots ,f_{j-1},g_j,f_{j+1},\cdots ,f_N)$, then the $j$-th $c$-transform of $G$ equals the $j$-th $c$-transform of the vector $F$.
But if $k\neq j$, the $k$-th $c$-transform of $G$ is not necessarily equal to the $k$-th $c$-transform of $F$.

\begin{lemma}\label{lm:c transforms are uniformly Lipschitz}
For any $f_j\in C(D_j)$ with $1\leq j\leq N$, the $j$th $c$-transform of the vector $\(f_1,\cdots ,f_N\)$
is a Lipschitz function with the constant $K$ from \eqref{eq:cost is Lipschitz}.
\end{lemma}
\begin{proof}
Given $y_j\in D_j$, since the spaces are compact there exist $\bar y_k\in D_k$ with $k\neq j$ such that
\[
f_j^c(y_j)=
c\(\bar y_1,\cdots ,\bar y_{j-1},y_j,\bar y_{j+1},\cdots ,\bar y_N\)
-
\sum_{k\neq j} f_k(\bar y_k).
\]
Hence
\begin{align*}
f_j^c(x_j)-f_j^c(y_j)
&=
f_j^c(x_j)
-
c\(\bar y_1,\cdots ,\bar y_{j-1},y_j,\bar y_{j+1},\cdots ,\bar y_N\)
-
\sum_{k\neq j} f_k(\bar y_k)\\
&\leq
\resizebox{0.85\textwidth}{!}{$c\(\bar y_1,\cdots ,\bar y_{j-1},x_j,\bar y_{j+1},\cdots ,\bar y_N\)-
\sum_{k\neq j} f_k(\bar y_k)
-
\(c\(\bar y_1,\cdots ,\bar y_{j-1},y_j,\bar y_{j+1},\cdots ,\bar y_N\)
-
\sum_{k\neq j} f_k(\bar y_k)\)$}\\
&=c\(\bar y_1,\cdots ,\bar y_{j-1},x_j,\bar y_{j+1},\cdots ,\bar y_N\)
-c\(\bar y_1,\cdots ,\bar y_{j-1},y_j,\bar y_{j+1},\cdots ,\bar y_N\)\\
&\leq
K\,d_j(x_j,y_j).
\end{align*}
Reversing the roles of $x_j$ and $y_j$ the desired result follows.
\end{proof}

\begin{definition}
The class of admissible functions is defined by 
\[
\mathcal K
=
\{(f_1,\cdots ,f_N)\in C(D_1)\times \cdots \times C(D_N):
f_1(x_1)+\cdots +f_N(x_N)\leq c(x_1,\cdots ,x_N)\,\forall x_j\in D_j, 1\leq j\leq N
\}.
\]
\end{definition}

\begin{lemma}\label{lm:the c of any is in K}
For any $(f_1,\cdots ,f_N)\in C(D_1)\times \cdots \times C(D_N)$ we have that 
\[
(f_1^c,f_2,\cdots ,f_N), (f_1,f_2^c,\cdots ,f_N),\cdots ,(f_1,f_2,\cdots ,f_N^c)\in \mathcal K.
\]
\end{lemma}
\begin{proof}
By definition of the $j$-th $c$-transform we can write
\begin{align*}
&f_1(x_1)+\cdots + f_{j-1}(x_{j-1})+ f_j^c(x_j)+f_{j+1}(x_{j+1})+\cdots +f_N(x_N)\\
&\leq
f_1(x_1)+\cdots + f_{j-1}(x_{j-1})+c(x_1,x_2,\cdots ,x_N)-\sum_{k\neq j}f_k(x_k)+f_{j+1}(x_{j+1})+\cdots +f_N(x_N)\\
&=c(x_1,x_2,\cdots ,x_N),
\end{align*}
that is $(f_1,\cdots , f_{j-1},f_j^c, f_{j+1},\cdots ,f_N)\in \mathcal K$ for $1\leq j\leq N$.
\end{proof}

\begin{lemma}\label{lm:fj less than or equal to fjc for f in K}
For any $(f_1,\cdots ,f_N)\in C(D_1)\times \cdots \times C(D_N)$ we have $(f_j^c)^c(x_j)= f_j^c(x_j)$,
and if $(f_1,\cdots ,f_N)\in \mathcal K$, then $f_j(x_j)\leq f_j^c(x_j)$ for $1\leq j\leq N$, and all $x_j\in D_j$.
\end{lemma}
\begin{proof}
If we consider the vector  
$\(f_1,\cdots , f_{j-1},f_j^c, f_{j+1},\cdots ,f_N\)$, then the $j$-th $c$ transform of this vector equals the $j$-th $c$ transform of the vector $\(f_1,\cdots , f_{j-1},f_j, f_{j+1},\cdots ,f_N\)$ because by definition
\[
(f_j^c)^c(x_j)=\inf_{z_k\in D_k, k\neq j}\left\{ c\(z_1,\cdots ,z_{j-1},x_j,z_{j+1},\cdots ,x_N\)-\sum_{k\neq j}f_k(x_k)\right\}
=f_j^c(x_j).
\]

If $(f_1,\cdots ,f_N)\in \mathcal K$, $f_j(x_j)\leq c(x_1,\cdots ,x_n)-\sum_{k\neq j} f_k(x_k)$ and taking infimum over $x_k\in D_k$ with $k\neq j$ the desired inequality follows.
\end{proof}

\begin{definition}
The vector $(f_1,\cdots ,f_N)\in C(D_1)\times \cdots \times C(D_N)$ is $c$-conjugate if $f_j=f_j^c$ for $1\leq j\leq N$.
\end{definition}


\begin{definition}\label{def:c normal mapping}
Given $(f_1,\cdots ,f_N)\in C(D_1)\times \cdots \times C(D_N)$ $c$-conjugate, the $c$-normal mapping (or $c$-superdifferential) of $f_i$ is given by 
\begin{align*}
\partial_cf_i(x_i)&=\left\{(z_1,\cdots ,z_{i-1},z_{i+1},\cdots ,z_N)\in D_1\times \cdots \times D_{i-1}\times D_{i+1}\times \cdots \times D_N:\right.\\
&\qquad f_1(z_1)+ \cdots + f_{i-1}(z_{i-1})+f_i(x_i)+ f_{i+1}(z_{i+1}) \cdots f_N(z_N)
=c(z_1,\cdots ,z_{i-1},x_i,z_{i+1},\cdots ,z_N)
\}.
\end{align*}
Clearly, $\partial_cf_i$ is a multi-valued map from $D_i$ into the subsets of $D_1\times \cdots \times D_{i-1}\times D_{i+1}\times \cdots \times D_N$.
\end{definition}

Let $\mu_1$ be a Borel measure in $D_1$. We introduce the following single-valued a.e. condition for the cost $c$:
\paragraph{\bf SV}\label{hyp:SV}
\textit{For each $(f_1,\cdots ,f_N)\in C(D_1)\times \cdots \times C(D_N)$ $c$-conjugate, the 
$c$-normal mapping 
\[
\partial_cf_1(x_1)=\left\{(z_2,\cdots ,z_N)\in D_2\times \cdots \times D_N:
 f_1(x_1)+ f_2(z_2)+\cdots + f_N(z_N)
=c(x_1,z_2,\cdots ,z_N)
\right\}
\]
is single-valued for a.e. $x_1\in D_1$ in the measure $\mu_1$.}


\begin{remark}\label{rmk:c normal map of c conjugate is always non empty}\rm
Notice that if $(f_1,\cdots ,f_N)$ is $c$-conjugate, then $\partial_cf_i(x_i)\neq \emptyset$ for each $x_i\in D_i$, $1\leq i\leq N$.
Because 
\begin{align*}
f_i(x_i)=f_1^c(x_i)&=
\inf_{z_k\in D_k,k\neq i}\(c\(z_1,\cdots ,z_{i-1},x_i, z_{i+1},\cdots ,z_N\)-\sum_{k\neq i}f_k(z_k) \)\\
&=c\(\bar z_1,\cdots ,\bar z_{i-1},x_i, \bar z_{i+1},\cdots ,\bar z_N\)-\sum_{k\neq i}f_k(\bar z_k)
\end{align*}
for some $\(\bar z_1,\cdots ,\bar z_{i-1}, \bar z_{i+1},\cdots ,\bar z_N\)\in D_1\times \cdots \times D_{i-1}\times D_{i+1}\times \cdots \times D_N$ and so $\(\bar z_1,\cdots ,\bar z_{i-1}, \bar z_{i+1},\cdots ,\bar z_N\)\in \partial_cf_i(x_i)$.\end{remark}

\begin{proposition}\label{prop:twist implies SV}
If $D_j$ are domains in $\R^n$, $|\partial D_1|=0$, $c\in C^1(D_1\times \cdots \times D_N)$, and the map $(x_2,\cdots ,x_N)\mapsto \nabla_{x_1}c(x_1,x_2,\cdots ,x_N)$ from $D_2\times \cdots \times D_N$ to $\R^n$ is injective for each $x_1\in D_1$, then the cost $c$ satisfies the single valued Lebesgue a.e. condition \nameref{hyp:SV}.
\end{proposition}

\begin{proof} Let $(f_1,\cdots ,f_N)$ be $c$-conjugate. Then by Lemma \ref{lm:c transforms are uniformly Lipschitz} and Rademacher's theorem  there exists $N\subset D_1$ of measure zero such that $f_1$ is differentiable in $D_1\setminus N$.
Let $x_1\in D_1\setminus (N\cup \partial D_1)$ 
and let $X=(x_2,\cdots ,x_N),Y=(y_2,\cdots ,y_N)$ be points in $\partial_cf_1(x_1)$.
Then
\[
f_1(x_1)+f_2(x_2)+\cdots +f_N(x_N)=c(x_1,x_2,\cdots ,x_N),
\]
and 
\[
f_1(x_1)+f_2(y_2)+\cdots +f_N(y_N)=c(x_1,y_2,\cdots ,y_N).
\]
Since $(f_1,\cdots ,f_N)$ is $c$-conjugate,
we have $f_N=f_N^c$ so 
\begin{align*}
&f_1(x_1)+f_2(x_2)+\cdots +f_{N-1}(x_{N-1})\\
&\qquad +\inf_{z_1\in D_1, \cdots z_{N-1}\in D_{N-1}} \(c(z_1,\cdots ,z_{N-1},x_N)-f_1(z_1)-\cdots -f_{N-1}(z_{N-1})\)=c(x_1,x_2,\cdots ,x_N), 
\end{align*}
and 
\begin{align*}
&f_1(x_1)+f_2(y_2)+\cdots +f_{N-1}(y_{N-1})\\
&\qquad +\inf_{z_1\in D_1, \cdots z_{N-1}\in D_{N-1}} \(c(z_1,\cdots ,z_{N-1},y_N)-f_1(z_1)-\cdots -f_{N-1}(z_{N-1})\)=c(x_1,y_2,\cdots ,y_N), 
\end{align*}
which implies
\[
f_1(x_1)+f(x_2)+\cdots +f_{N-1}(x_{N-1})-c(x_1,x_2,\cdots ,x_N)
\geq 
f_1(z_1)+\cdots +f_{N-1}(z_{N-1})-c(z_1,\cdots ,z_{N-1},x_N) 
\] 
and 
\[
f_1(x_1)+f(y_2)+\cdots +f_{N-1}(y_{N-1})-c(x_1,y_2,\cdots ,y_N)
\geq 
f_1(z_1)+\cdots +f_{N-1}(z_{N-1})-c(z_1,\cdots ,z_{N-1},y_N) 
\] 
for all $z_1\in D_1,\cdots , z_{N-1}\in D_{N-1}$.
It follows that the maximum of 
\[
f_1(z_1)+\cdots +f_{N-1}(z_{N-1})-c(z_1,\cdots ,z_{N-1},x_N)
\] 
in $D_1\times \cdots \times D_{N-1}$ is attained at $z_i=x_i$, $1\leq i\leq N-1$, 
and the maximum of 
\[
f_1(z_1)+\cdots +f_{N-1}(z_{N-1})-c(z_1,\cdots ,z_{N-1},y_N)
\] 
in $D_1\times \cdots \times D_{N-1}$ is attained at $z_1=x_1$ and $z_i=y_i$, $2\leq i\leq N-1$.

Since $c$ is $C^1$ and $x_1$ is a differentiability point of $f_1$ in the interior of $D_1$ we get 
\[
\nabla_{z_1}f_1(x_1)-\nabla_{z_1}c(x_1,x_1,\cdots ,x_{N-1},x_N)=0,  
\]
and 
\[
\nabla_{z_1}f_1(x_1)-\nabla_{z_1}c(x_1,y_2,\cdots ,y_{N-1},y_N)=0  
\]
and since the map $(x_2,\cdots ,x_N)\mapsto \nabla_{x_1}c(x_1,x_2,\cdots ,x_N)$ is injective we obtain $X=Y$ as desired.

\end{proof}
\begin{lemma}
Suppose the cost $c$ satisfies \eqref{eq:cost is Lipschitz} and the single-valued a.e. condition \nameref{hyp:SV}.
Given $(f_1,\cdots ,f_N)\in C(D_1)\times \cdots \times C(D_N)$ $c$-conjugate, consider the class 
\[
\mathcal C=\left\{E\subset D_2\times \cdots \times D_N: \(\partial_cf_1\)^{-1}(E) \text{ is a $\mu_1$-measurable subset of $D_1$} \right\}.\footnote{$\(\partial_cf_1\)^{-1}(E)=\{x_1\in D_1: \partial_cf_1(x_1)\cap E\neq \emptyset\}$.}
\]
Then the class $\mathcal C$ is a $\sigma$-algebra containing all Borel subsets of $D_2\times \cdots \times D_N$. 
\end{lemma}
\begin{proof}
It is clear that $\mathcal C$ is closed by countable unions, and  $D_2\times \cdots \times D_N\in \mathcal C$ from Remark \ref{rmk:c normal map of c conjugate is always non empty}. That $\mathcal C$ is closed by complements follows from the single-valued a.e condition \nameref{hyp:SV} writing for $E\in \mathcal C$
\[
\(\partial_cf_1\)^{-1}(E^c)=\(\(\partial_cf_1\)^{-1}(E)\)^c \cup \(\(\partial_cf_1\)^{-1}(E^c)\cap 
\(\partial_cf_1\)^{-1}(E)\).
\] 
To prove $\mathcal C$ contains the Borel sets it is enough to show that $\(\partial_cf_1\)^{-1}(K)$ is compact for each compact $K\subset D_2\times \cdots \times D_N$. 
Let $\{u_j\}_{j=1}^\infty$ be a sequence in $\(\partial_cf_1\)^{-1}(K)\subset D_1$. 
Since $D_1$ is compact, we may assume through a subsequence that $u_j\to u^0$ for some $u^0\in D_1$.
By Remark \ref{rmk:c normal map of c conjugate is always non empty} there is $(v_2^j,\cdots ,v_N^j)\in \partial_cf_1(u_j)\subset K$ and since $K$ is compact there is a subsequence $(v_2^{j_\ell},\cdots ,v_N^{j_\ell})\to (v_2^0,\cdots ,v_N^0)\in K$.
We then have 
\[
f_1\(u_{j_\ell}\)+f_2\(v_2^{j_\ell}\)+\cdots +f_N\(v_N^{j_\ell}\)
=
c\(u_\ell, v_2^{j_\ell}, \cdots , v_N^{j_\ell}\)
\] 
and since $(f_1,\cdots ,f_N)\in C(D_1)\times \cdots \times C(D_N)$ letting $\ell\to \infty$ yields 
\[
f_1(u^0)+f_2(v_2^0)+\cdots +f_N(v_N^0)
=
c\(u^0,v_2^0, \cdots ,v_N^0\),
\]
that is, $\(v_2^0, \cdots ,v_N^0\)\in \partial_cf_1(u^0)$ and so $u^0\in (\partial_cf_1)^{-1}(K)$.

\end{proof}

\setcounter{equation}{0}
\section{Optimization of the Kantorovich functional $I$}
We have finite Borel measures $\mu_j$ on $D_j$, $1\leq j\leq N$, and define the linear functional 
\[
I(f_1,\cdots ,f_N)
=
\int_{D_1}f_1(x_1)\,d\mu_1(x_1) +\cdots +\int_{D_N}f_N(x_N)\,d\mu_N(x_N).
\] 
\begin{definition}
Let $2\leq j\leq N$ and let $T_j$ be a multi-valued map from $D_1$ to $D_j$.
We say that $T_j$ is measure preserving $(\mu_1,\mu_j)$ if $T_j^{-1}(E)$ is a $\mu_1$-measurable subset of $D_1$ and
\[
\mu_1\(T_j^{-1}(E)\)=\mu_j(E)
\]
for each Borel set $E\subset D_j$.
\end{definition}
We have the following.
\begin{lemma}\label{lm:measure preserving implies maximizer}
Let $(f_1,\cdots ,f_N)\in C(D_1)\times \cdots \times C(D_N)$ be $c$-conjugate and let $P_j$ be the projection from $D_2\times \cdots \times D_N$ to $D_j$, $2\leq j\leq N$.
Given $x_1\in D_1$ define $T_j x_1=P_j\(\partial_c f_1(x_1)\)$ and assume the cost $c$ satisfies the single valued a.e. condition \nameref{hyp:SV}.

If $T_j$ is measure preserving $(\mu_1,\mu_j)$ for $2\leq j\leq N$, then $(f_1,\cdots ,f_N)$ maximizes the functional $I$ over $\mathcal K$.
\end{lemma}

\begin{proof}
Let $(g_1,\cdots ,g_N)\in \mathcal K$.
Since $\partial_c f_1$ is single valued a.e., there exists a set $S\subset D_1$ such that $\partial_c f_1(x_1)$ is a singleton for all $x_1\in D_1\setminus S$ and $\mu_1(S)=0$.
So for $x_1\in D_1\setminus S$ we have
\begin{align*}
&g_1(x_1)+g_2\(P_2\(\partial_c f_1(x_1)\)\) +\cdots + g_N\(P_N\(\partial_c f_1(x_1)\)\)\\
&\qquad 
\leq  
c\(x_1,P_2\(\partial_c f_1(x_1)\),\cdots , P_N\(\partial_c f_1(x_1)\)\)\\
&\qquad =f_1(x_1)+f_2\(P_2\(\partial_c f_1(x_1)\)\) +\cdots + f_N\(P_N\(\partial_c f_1(x_1)\)\).
\end{align*}
Integrating this inequality over $D_1\setminus S$ yields
\begin{align*}
&\int_{D_1}g_1(x_1)\,d\mu_1(x_1)+\int_{D_1} g_2\(P_2\(\partial_c f_1(x_1)\)\)\,d\mu_1(x_1) +\cdots +\int_{D_1} g_N\(P_N\(\partial_c f_1(x_1)\)\)\,d\mu_1(x_1)\\
&\qquad \leq \int_{D_1}f_1(x_1)\,d\mu_1(x_1)+\int_{D_1} f_2\(P_2\(\partial_c f_1(x_1)\)\)\,d\mu_1(x_1) +\cdots +\int_{D_1} f_N\(P_N\(\partial_c f_1(x_1)\)\)\,d\mu_1(x_1).
\end{align*}
Since $T_j$ is measure preserving $(\mu_1,\mu_j)$, $2\leq j\leq N$, we have from \cite[Lemma 5.4]{2023-gutierrez:bookonOTandoptics} that
\begin{align*}
\int_{D_1} g_j\(P_j\(\partial_c f_1(x_1)\)\)\,d\mu_1(x_1)&=\int_{D_j}g_j(x_j)\,d\mu_j(x_j) \\
\int_{D_1} f_j\(P_j\(\partial_c f_1(x_1)\)\)\,d\mu_1(x_1)&=\int_{D_j}f_j(x_j)\,d\mu_j(x_j)
\end{align*}
and therefore $I(g_1,\cdots ,g_N)\leq I(f_1,\cdots ,f_N)$.
\end{proof}

We have the following converse to Lemma \ref{lm:measure preserving implies maximizer}.

\begin{lemma}
If $(\phi_1,\cdots ,\phi_N)\in \mathcal K$ is $c$-conjugate and maximizes the functional $I$ over $\mathcal K$, $c$ satisfies the single-valued a.e. condition \nameref{hyp:SV}, and   
$T_j x_1=P_j\(\partial_c \phi_1(x_1)\)$, then
$T_j$ is measure preserving $(\mu_1,\mu_j)$ for $2\leq j\leq N$.

\end{lemma}
\begin{proof}
Fix $2\leq j\leq N$, take $v_j\in C(D_j)$ and $\theta\in \R$, and consider
\[
\(\phi_1,\cdots ,\phi_{j-1}, \phi_j+\theta v_j, \phi_{j+1},\cdots , \phi_N\).
\]
By Lemma \ref{lm:the c of any is in K}
\[
\Phi(\theta):=\((\phi_1)^c_\theta,\cdots ,\phi_{j-1}, \phi_j+\theta v_j, \phi_{j+1},\cdots , \phi_N\)\in \mathcal K,
\]
where 
\[
(\phi_1)^c_\theta(x_1)
=
\inf_{z_k\in D_k, k\neq 1}
\left\{
c(x_1,z_2,\cdots ,z_N)-\(\sum_{k=2}^{j-1} \phi_k(x_k)+\phi_j(x_j)+\theta v_j(x_j)+\sum_{k=j+1}^N \phi_k(x_k)\)\right\}.
\]
Since $(\phi_1,\cdots ,\phi_N)$ is a maximizer of $I$, we have  
\begin{equation}\label{eq:I theta less than or equal the I mazimizer}
I\(\Phi(\theta)\)\leq I(\phi_1,\cdots ,\phi_N)
\end{equation}
for each $\theta\in \R$.
We claim that 
\[
\lim_{\theta\to 0} \dfrac{I\(\Phi(\theta)\)- I(\phi_1,\cdots ,\phi_N)}{\theta}
=-\int_{D_1} v_j\(P_j\(\partial_c \phi_1(x_1)\)\)\,d\mu_1(x_1)
+\int_{D_j} v_j(x_j)\,d\mu_j(x_j).
\]
If the claim holds, then from \eqref{eq:I theta less than or equal the I mazimizer} the limit must be zero and so
\begin{equation}\label{eq:measure preserving of partial c}
\int_{D_1} v_j\(P_j\(\partial_c \phi_1(x_1)\)\)\,d\mu_1(x_1)
=\int_{D_j} v_j(x_j)\,d\mu_j(x_j)
\end{equation}
for any $v_j\in C(D_j)$. Now using \cite[Lemma 5.4]{2023-gutierrez:bookonOTandoptics} we obtain that $T_j=P_j\partial_c f_1$ is measure preserving $(\mu_1,\mu_j)$. 

To prove the claim write
\[
\dfrac{I\(\Phi(\theta)\)- I(\phi_1,\cdots ,\phi_N)}{\theta}
=
\int_{D_1} \dfrac{(\phi_1)^c_\theta(x_1)-\phi_1(x_1)}{\theta}\,d\mu_1(x_1)
+
\int_{D_j} v_j(x_j)\,d\mu_j(x_j)
\]
and we shall prove that $\dfrac{(\phi_1)^c_\theta(x_1)-\phi_1(x_1)}{\theta}$ is bounded uniformly in $\theta$ for $|\theta|$ bounded and that 
\begin{equation}\label{eq:limit of difference equal subdifferential ae}
\lim_{\theta\to 0}\dfrac{(\phi_1)^c_\theta(x_1)-\phi_1(x_1)}{\theta}=-v_j\(P_j\(\partial_c \phi_1(x_1)\)\)
\end{equation}
for a.e. $x_1\in D_1$. Hence by Lebesgue dominated convergence theorem the claim follows.
By compactness 
\begin{equation}\label{eq:writing of phi1c theta with ztheta}
(\phi_1)^c_\theta(x_1)
=
c\(x_1,z_{2,\theta},\cdots ,z_{N,\theta}\)-\(\sum_{k=2}^{j-1} \phi_k(z_{k,\theta})+\phi_j(z_{j,\theta})+\theta v_j(z_{j,\theta})+\sum_{k=j+1}^N \phi_k(z_{k,\theta})\)
\end{equation}
for some $\(z_{2,\theta},\cdots ,z_{N,\theta} \)\in D_2\times \cdots \times D_N$ depending on $\theta$ and $x_1$.
Since $(\phi_1,\cdots ,\phi_N)\in \mathcal K$ is $c$-conjugate, we have $\phi_1^c=\phi_1$ and again by compactness
\begin{equation}\label{eq:formula for phi_1^c with infimum attained}
\phi_1^c(x_1)
=
c\(x_1,\bar z_2,\cdots ,\bar z_N\)-\sum_{k=2}^N \phi_k(\bar z_k),
\end{equation}
for some $\bar z_j\in D_j$, $j\neq 1$, depending on $x_1$.

Hence by definition of $\phi_1^c$
\[
(\phi_1)^c_\theta(x_1)-\phi_1(x_1)=(\phi_1)^c_\theta(x_1)-\phi_1^c(x_1)
\geq 
-\theta \,v_j\(z_{j,\theta}\).
\]
Similarly, from \eqref{eq:formula for phi_1^c with infimum attained} it follows that
\[
(\phi_1)^c_\theta(x_1)-\phi_1(x_1)=(\phi_1)^c_\theta(x_1)-\phi_1^c(x_1)
\leq 
-\theta \,v_j\(\bar z_j\).
\]
Therefore
\[
-\theta \,v_j\(z_{j,\theta}\)\leq (\phi_1)^c_\theta(x_1)-\phi_1(x_1)\leq -\theta\,v_j\(\bar z_j\) 
\]
which implies, since $v_j\in C(D_j)$, that the ratio $\dfrac{(\phi_1)^c_\theta(x_1)-\phi_1(x_1)}{\theta}$ is bounded uniformly in $\theta$ for $|\theta|$ bounded.

It remains to prove \eqref{eq:limit of difference equal subdifferential ae}.
In fact, since $T_j x_1=P_j\(\partial_c \phi_1(x_1)\)$ is single valued for a.e. there exists a set $S\subset D_1$ with $\mu_1(S)=0$ such that $T_j x_1=P_j\(\partial_c \phi_1(x_1)\)$ is a singleton for $x_1\in D_1\setminus S$.
From \eqref{eq:formula for phi_1^c with infimum attained} and since $\phi_1^c=\phi_1$, it follows that $(\bar z_2,\cdots ,\bar z_N)=\partial_c \phi_1(x_1)$ for $x_1\in D_1\setminus S$.
We claim that $\(z_{2,\theta},\cdots ,z_{N,\theta} \)\to (\bar z_2,\cdots ,\bar z_N)$ when $\theta\to 0$.
Suppose by contradiction this were not true. Then by compactness there exists a sequence $\theta_k\to 0$ and a point $(z_2^*,\cdots ,z_N^*)\neq (\bar z_2,\cdots ,\bar z_N)$ such that $\(z_{2,\theta_k},\cdots ,z_{N,\theta_k} \)\to (z_2^*,\cdots ,z_N^*)$.
Then by \eqref{eq:writing of phi1c theta with ztheta}
\[
\phi_1^c(x_1)=\lim_{\theta_k \to 0} (\phi_1)^c_{\theta_k}(x_1)
=
c\(x_1,z_2^*,\cdots ,z_N^*\)-\sum_{k\neq 1}\phi_k(z_k^*).
\]
Again, since $\phi_1^c(x_1)=\phi(x_1)$, this implies that $(z_2^*,\cdots ,z_N^*)\in \partial_c \phi_1(x_1)$, that is, $\partial_c \phi_1(x_1)$ would contain two different points which yields a contradiction for $x_1\in D_1\setminus S$.
\end{proof}

\begin{theorem}
Suppose $\mu_1(D_1)=\mu_j(D_j)$ for $2\leq j\leq N$. 
There exists $(\phi_1,\cdots ,\phi_N)\in C(D_1)\times \cdots \times C(D_N)$ $c$-conjugate such that 
\[
I\(\phi_1,\cdots ,\phi_N\)= \sup \left\{I\(u_1,\cdots ,u_N\): \(u_1,\cdots ,u_N\)\in \mathcal K\right\}:=I_0\]
\end{theorem}
\begin{proof}

We begin with the following construction. 
Let $(u_1,\cdots ,u_N)\in \mathcal K$.

Step 1. By Lemma \ref{lm:the c of any is in K} $(u_1^c,u_2,\cdots ,u_N)\in \mathcal K$
and by Lemma \ref{lm:fj less than or equal to fjc for f in K} $u_1\leq u_1^c$.

Step 2. Take the $c$-transform of the vector $(u_1^c,u_2,\cdots ,u_N)$ and take the second coordinate of that vector denoted by $u_2^c$. Then by Lemma \ref{lm:the c of any is in K} $(u_1^c,u_2^c,\cdots ,u_N)\in \mathcal K$
and by Lemma \ref{lm:fj less than or equal to fjc for f in K} $u_2\leq u_2^c$.

Step 3. Next take the $c$-transform of the vector $(u_1^c,u_2^c, u_3, \cdots ,u_N)$ and take the third coordinate of that vector denoted by $u_3^c$. Then by Lemma \ref{lm:the c of any is in K} $(u_1^c,u_2^c,u_3^c,\cdots ,u_N)\in \mathcal K$
and by Lemma \ref{lm:fj less than or equal to fjc for f in K} $u_3\leq u_3^c$.

Continuing in this way we construct a vector $(u_1^c,u_2^c,u_3^c,\cdots ,u_N^c)\in \mathcal K$
with $u_j\leq u_j^c$ for $1\leq j \leq N$.

Now to the proof. There exists a sequence $U_k=\(u_1^k,\cdots ,u_N^k\)\in \mathcal K$ such that $I(U_k)\to I_0$ as $k\to \infty$.
Denote by $\phi_j^k=(u_j^k)^c$ for $1\leq j\leq N$ the components of the vector constructed relative to $U_k$. 
By construction $\(\phi_1^k,\cdots ,\phi_N^k\)\in \mathcal K$ and  $u_j^k\leq \phi_j^k$ for $1\leq j\leq N$.
Hence $I(U_k)\leq I\(\phi_1^k,\cdots , \phi_N^k\)\leq I_0$ and so $I\(\phi_1^k,\cdots , \phi_N^k\)\to I_0$ as $k\to \infty$.

Let $\alpha_j^k=\min_{x_j\in D_j} \phi_j^k(x_j)$ for $1\leq j\leq N-1$ and define
\[
\bar v_j^k (x_j)= \phi_j^k(x_j)-\alpha_j^k,\text{ for $1\leq j\leq N-1$ and $\bar v_N^k (x_N)=\phi_N^k(x_N)+\sum_{j=1}^{N-1}\alpha_j^k$}.
\]
Clearly, $\(\bar v_1^k ,\cdots ,\bar v_N^k\)\in \mathcal K$ and $I\(\bar v_1^k ,\cdots ,\bar v_N^k\)=I\(\phi_1^k,\cdots , \phi_N^k\)$ since $\mu_1(D_1)=\mu_j(D_j)$ for $2\leq j\leq N$. 
Also $\bar v_j^k (x_j)$ are uniformly bounded in $k$ for all $1\leq j\leq N-1$ because
\[
\bar v_j^k (x_j)= \phi_j^k(x_j)-\alpha_j^k=\phi_j^k(x_j)-\phi_j^k(z_j)
\]
for some $z_j\in D_j$ so $|\bar v_j^k (x_j)|\leq |\phi_j^k(x_j)-\phi_j^k(z_j)|\leq K\,d_j(x_j,z_j)\leq K\,\diam(D_j)\leq K_0$ by Lemma \ref{lm:c transforms are uniformly Lipschitz}.
In addition, the functions $\bar v_j^k$ are Lipschitz with a constant bounded uniformly in $k$ for $1\leq j\leq N-1$ by Lemma \ref{lm:c transforms are uniformly Lipschitz}.
 
Now take the $c$-transform of the vector $\(\bar v_1^k ,\cdots ,\bar v_N^k\)$, then the $N$-component of this transform is
\[
(\bar v_N^k)^c(x_N)=\inf_{z_k\in D_k,k\neq N}\left\{c\(z_1,\cdots,z_{N-1},x_N\)-\sum_{j=1}^{N-1}\bar v_j^k (z_j) \right\}
\] 
and since the $\bar v_j^k$ are uniformly bounded in $k$ for $1\leq j\leq  N-1$ and $c$ is bounded it follows that $(\bar v_N^k)^c$ is uniformly bounded in $k$. Also $(\bar v_N^k)^c$ is Lipschitz with a constant bounded uniformly in $k$ by Lemma \ref{lm:c transforms are uniformly Lipschitz}.
Since $\(\bar v_1^k ,\cdots ,\bar v_N^k\)\in \mathcal K$, by Lemma \ref{lm:fj less than or equal to fjc for f in K} $\bar v_N^k\leq (\bar v_N^k)^c$, and so by Lemma \ref{lm:the c of any is in K}
$\(\bar v_1^k ,\cdots ,\bar v_{N-1}^k,(\bar v_N^k)^c\)\in \mathcal K$. 
Therefore 
\[
I\(\phi_1^k,\cdots , \phi_N^k\)=I\(\bar v_1^k ,\cdots ,\bar v_N^k\)\leq I\(\bar v_1^k ,\cdots ,\bar v_{N-1}^k,(\bar v_N^k)^c\)
\]
obtaining that $I\(\bar v_1^k ,\cdots ,\bar v_{N-1}^k,(\bar v_N^k)^c\)\to I_0$.

We have then constructed the sequence $\(\bar v_1^k ,\cdots ,\bar v_{N-1}^k,(\bar v_N^k)^c\)$ in $\mathcal K$ that is uniformly bounded and equicontinuous, therefore by Arzel\`a-Ascoli's theorem it contains a subsequence uniformly convergent to some vector
$(f_1,\cdots ,f_N)$. Clearly, $(f_1,\cdots ,f_N)\in \mathcal K$ and $I(f_1,\cdots ,f_N)=I_0$.

Once again applying the construction at the beginning to $(f_1,\cdots ,f_N)$ we obtain a vector  
$(f_1^c,\cdots ,f_N^c)\in \mathcal K$ with $f_j\leq f_j^c$ so we have $I(f_1^c,\cdots ,f_N^c)\geq 
I\(f_1,\cdots ,f_N\)$. That is, $(f_1^c,\cdots ,f_N^c)$ is a maximizer of $I$ and by Lemma \ref{lm:fj less than or equal to fjc for f in K} is $c$-conjugate.
\end{proof}

\setcounter{equation}{0}
\section{Existence and uniqueness of solutions to the multi-marginal Monge problem}
\begin{theorem}\label{thm:existence of solutions to Monge problem}
Assume the cost $c$ satisfies \nameref{hyp:SV}, and $\{(D_j,d_j)\}_{j=1}^N$ are compact metric spaces with $\mu_j$ Borel measures satisfying $\mu_1(D_1)=\mu_j(D_j)$ for $2\leq j \leq N$.

Let $\mathcal S$ be the class of measure preserving maps $s:D_1\to D_2\times \cdots \times D_N$, that is, 
if $P_j$ is the projection of $D_2\times \cdots \times D_N$ onto $D_j$ the mapping $T_j=P_j\circ s$ is measure preserving $(\mu_1,\mu_j)$ for $2\leq j\leq N$, i.e., for each Borel set $E\subset D_j$, $T_j^{-1}(E)$ is a Borel subset of $D_1$ and
\[
\mu_1\(T_j^{-1}(E)\)=\mu_j(E).
\]

If $(f_1,\cdots ,f_N)\in \mathcal K$ is $c$-conjugate and maximizes the functional $I$ over $\mathcal K$, then  
\[
\inf \left\{\int_{D_1} c\(x_1, s(x_1)\)\,d\mu_1(x_1): s\in \mathcal S  \right\}
\]
is attained at $s=\partial_c f_1$.
Moreover 
\begin{equation}\label{eq:infimum monge equals max functional}
\inf \left\{\int_{D_1} c\(x_1, s(x_1)\)\,d\mu_1(x_1): s\in \mathcal S  \right\}
=\sup \{I(u_1,\cdots ,u_N): (u_1,\cdots ,u_N)\in \mathcal K\}.
\end{equation}
\end{theorem}

\begin{proof}
Let $s\in \mathcal S$ and write $s(x_1)=\(s_2(x_1),\cdots ,s_N(x_1)\)$. Since $(f_1,\cdots ,f_N)\in \mathcal K$ and $s$ is measure preserving
\begin{align*}
&\int_{D_1} c\(x_1, s(x_1)\)\,d\mu_1(x_1)\\
&\geq
\int_{D_1}f_1(x_1)\,d\mu_1(x_1)+\int_{D_1} f_2(s_2(x_1))\,d\mu_1(x_1)+\cdots +\int_{D_1} f_N(s_N(x_1))\,d\mu_1(x_1)\\
&=\int_{D_1}f_1(x_1)\,d\mu_1(x_1)
+\int_{D_2} f_2(x_2)\,d\mu_2(x_2)
+\cdots 
+\int_{D_N} f_N(x_N)\,d\mu_N(x_N)\\
&=\int_{D_1} f_1(x_1)\,d\mu_1(x_1)
+\int_{D_1} f_2\(P_2\(\partial_c f_1(x_1)\)\)\,d\mu_1(x_1)
+\cdots +
\int_{D_1} f_N \(P_N\(\partial_c f_1(x_1)\)\)\,d\mu_1(x_1)
\quad \text{by equation \eqref{eq:measure preserving of partial c}}\\
&=\int_{D_1} \(f_1(x_1)+f_2\(P_2\(\partial_c f_1(x_1)\)\)
+\cdots +
f_N \(P_N\(\partial_c f_1(x_1)\)\)\)\,d\mu_1(x_1)\\
&=
\int_{D_1} c\(x_1,P_2\(\partial_c f_1(x_1)\),\cdots ,P_N\(\partial_c f_1(x_1)\)\)\,d\mu_1(x_1)
\end{align*}
by definition of $c$-normal mapping.
\end{proof}

\begin{theorem}\label{thm:uniqueness of the  multimarginal problem}
Suppose $\mu_1(G)>0$ for each open set $G\subset D_1$ and the cost $c(x_1,\cdots ,x_N)$ satisfies the single valued a.e. condition \nameref{hyp:SV}. Then the mapping minimizing 
\[
\int_{D_1} c\(x_1, s(x_1)\)\,d\mu_1(x_1)
\]
over all measure preserving maps $s\in \mathcal S$ is unique.
\end{theorem}
\begin{proof}
From Theorem \ref{thm:existence of solutions to Monge problem}, let $\partial_c f_1$ be the $c$-normal mapping associated with the $c$-conjugate function $(f_1,\cdots ,f_N)\in \mathcal K$ that minimizes the functional $I$ over $\mathcal K$.
Let $s_0(x_1)=\(s_2(x_1),\cdots ,s_N(x_1)\)$ (defined a.e.) be with $s_j$ measure preserving $(\mu_1,\mu_j)$ that also minimizes the integral.
Since $(f_1,\cdots ,f_N)\in \mathcal K$, we have 
\[
f_1(x_1)+f_2(s_2(x_1))+\cdots +f_N(s_N(x_1))\leq c\(x_1,s_0(x_1)\)
\]
and hence
\begin{align*}
0
&\leq
\int_{D_1}\left\{c\(x_1,s_0(x_1)\)-\(f_1(x_1)+f_2(s_2(x_1))+\cdots +f_N(s_N(x_1))\) \right\}\,d\mu_1\\
&=\int_{D_1}c\(x_1,s_0(x_1)\)\,d\mu_1
-
\int_{D_1}f_1(x_1)\,d\mu_1 - \int_{D_1} f_2(s_2(x_1))\,d\mu_1-\cdots - \int_{D_1} f_N(s_N(x_1))\,d\mu_1\\
&=\int_{D_1}c\(x_1,s(x_1)\)\,d\mu_1
-
\int_{D_1}f_1(x_1)\,d\mu_1 - \int_{D_2} f_2(x_2)\,d\mu_2-\cdots - \int_{D_N} f_N(x_N)\,d\mu_N\\
&=\inf \left\{\int_{D_1} c\(x_1, s(x_1)\)\,d\mu_1(x_1): s\in \mathcal S  \right\}
-
\int_{D_1}f_1(x_1)\,d\mu_1 - \int_{D_2} f_2(x_2)\,d\mu_2-\cdots - \int_{D_N} f_N(x_N)\,d\mu_N\\
&\leq 0
\end{align*}
since $s_i$ is measure preserving and from \eqref{eq:infimum monge equals max functional}.
Since $\mu_1(G)>0$ for each $G$ open it follows that  
\[
\(f_1(x_1)+f_2(s_2(x_1))+\cdots +f_N(s_N(x_1))\)=c\(x_1,s_0(x_1)\)
\]
for $x_1\in D_1\setminus N$ with $\mu_1(N)=0$. That is, $s_0(x_1)\in \partial_cf_1(x_1)$ for all $x_1\in D_1\setminus N$. But from the single valued assumption on the $c$-normal mapping there is $N'\subset D_1$ such that $\mu_1(N')=0$ with $\partial_cf_1(x_1)$ a singleton for all $x_1\in D_1\setminus N'$.
Therefore, for each $x_1\in D_1\setminus (N\cup N')$ we obtain $s_0(x_1)=\partial_cf_1(x_1)$ proving the uniqueness.

\end{proof}

As a corollary we obtain the following representation formula.

\begin{theorem}\label{thm:representation formula for multimarginal}
Suppose $D_i\subset \R^n$ are compact domains, $\mu_i$ are Borel measures in $D_i$, $\mu_i(D_i)=\mu_1(D_1)$, $1\leq i\leq N$, $\mu_1(\partial D_1)=0$, $\mu_1(G)>0$ for each $G\subset D_1$ open, and 
$c:D_1\times \cdots \times D_N\to \R$ is a $C^1$ cost function such that the map $(x_2,\cdots ,x_N)\mapsto \nabla_{x_1}c(x_1,x_2,\cdots ,x_N)$ is injective for each $x_1\in D_1$.
If $(f_1,\cdots ,f_N)\in \mathcal K$ is $c$-conjugate and a maximizer of the functional $I$ over $\mathcal K$, then the optimal map $T_2,\cdots ,T_N$ for the multi-marginal Monge problem satisfies 
\begin{equation}\label{eq:representation formula for multimarginal}
\nabla_{x_1}f_1(x_1)=\nabla_{x_1}c\(x_1,T_2x_1,\cdots ,T_Nx_1\)
\end{equation} 
for a.e. $x_1\in D_1$.
\end{theorem}

\begin{proof}
From the assumption on $c$ it follows from Propositionl \ref{prop:twist implies SV} that $c$ satisfies \nameref{hyp:SV} and therefore by application of Theorems \ref{thm:existence of solutions to Monge problem} and \ref{thm:uniqueness of the  multimarginal problem}, the optimal map $(T_2,\cdots ,T_N)$ for the multi-marginal Monge problem is the $c$-normal mapping of the function $f_1$, i.e., $\partial_c f_1(x_1)=(T_2x_1,\cdots ,T_Nx_1)$ for  $x_1\in D_1\setminus N$ with $\mu_1(N)=0$.
For $x_1\in D_1\setminus N$, let $(x_2,x_3, \cdots ,x_N)$ be the unique point in $\partial_c f_1(x_1)$. 
Then
\[
f_1(x_1)+f_2(x_2)+\cdots +f_N(x_N)=c(x_1,x_2,\cdots ,x_N).
\]
Since $(f_1,\cdots ,f_N)\in \mathcal K$ is $c$-conjugate,
we have $f_N=f_N^c$ so 
\begin{align*}
&f_1(x_1)+\cdots +f_{N-1}(x_{N-1})\\
&\qquad +\inf_{z_1\in D_1, \cdots z_{N-1}\in D_{N-1}} \(c(z_1,\cdots ,z_{N-1},x_N)-f_1(z_1)-\cdots -f_{N-1}(z_{N-1})\)=c(x_1,x_2,\cdots ,x_N), 
\end{align*}
which implies
\[
f_1(x_1)+\cdots +f_{N-1}(x_{N-1})+
c(z_1,\cdots ,z_{N-1},x_N)-f_1(z_1)-\cdots -f_{N-1}(z_{N-1})\geq c(x_1,x_2,\cdots ,x_N)
\]
for all $z_1\in D_1,\cdots , z_{N-1}\in D_{N-1}$
that is,
\[
f_1(x_1)+\cdots +f_{N-1}(x_{N-1})-c(x_1,x_2,\cdots ,x_N)
\geq 
f_1(z_1)+\cdots +f_{N-1}(z_{N-1})-c(z_1,\cdots ,z_{N-1},x_N) 
\] 
It follows that the maximum of 
\[
f_1(z_1)+\cdots +f_{N-1}(z_{N-1})-c(z_1,\cdots ,z_{N-1},x_N)
\] 
in $D_1\times \cdots \times D_{N-1}$ is attained at $z_i=x_i$, $1\leq i\leq N-1$. Since $c$ is $C^1$ and assuming $x_1$ is a differentiability point of $f_1$ in the interior of $D_1$ we get 
\[
\nabla_{x_1}f_1(z_1)-\nabla_{z_1}c(z_1,\cdots ,z_{N-1},x_N)=0,  \text{ when $z_i=x_i$, $1\leq i\leq N_1$}
\]
which proves the theorem.

\end{proof}

\section{Application of the Multi marginal Monge problem in Optics}
\label{subsec:optics application of multivariate monge}

In this section, we present an application of the multi-marginal Monge problem to optics. More specifically, our goal is to design a metalens that simultaneously refracts and reflects light waves to achieve prescribed energy distributions.

Metalenses, or metasurfaces, are ultra-thin optical devices constructed from nano-structures designed to focus light for imaging purposes. They can be modeled as a pair $(\Gamma, \Phi)$, where $\Gamma$ is a surface in three-dimensional space given by the graph of a $C^2$ function $f$, and $\Phi$ is a $C^1$ function defined in a neighborhood of $\Gamma$, known as the phase discontinuity.

Using such a phase, the nano-structures are engineered by adjusting their shape, size, position, and orientation—typically arranged on the surface as tiny pillars, rings, or other geometric configurations. These structures work together to manipulate light waves as they pass through the surface. Unlike conventional lenses, which rely on a gradual accumulation of phase as the wave propagates through the bulk of the lens, metasurfaces introduce abrupt phase changes along the optical path to reshape the scattered wave.

This topic is an active area of research in engineering and has significant potential for applications in optics. For further discussion in the engineering literature, see for example \cite{science-runner-ups-2016}, \cite{2107planaroptics:capasso}, and \cite{Chen_broadband}.

We conclude this introduction by outlining the organization of the section. In Section \ref{subsec:generalized Snell law}, we introduce the generalized Snell’s laws for refraction and reflection. Then, in Section \ref{subsec:design of reflecting-refracting metasurfaces}, we describe and solve the optical design problem under consideration and explain its connection to the multi-marginal Monge problem. The associated cost function is analyzed in Section \ref{analysisofmultimarginalcost}, where we show that, under suitable assumptions on the parameters, the cost satisfies injectivity conditions. These allow us to apply results from previous sections to conclude the existence of solutions to our optical design problem, Theorem \ref{thm:main result representation of the phase}.

\subsection{Generalized Snell's Law}\label{subsec:generalized Snell law}

Refraction in metalenses follow the generalized Snell's law which gives a relationship between the direction of the incident, and the refracted (or transmitted) light rays, the refractive indices of the media where these rays travel, and the gradient of the phase function. 
The law is precisely described as follows. 
Let $\Gamma$ be a surface in $\mathbb{R}^3$ as above that separates two media, with refractive indices $n_1$ and $n_2$, respectively. Let the incoming unit direction of a ray striking the surface $\Gamma$ at $(x, f(x))$ be $\textbf{e}(x)$, and the unit direction of {\it refracted or transmitted} ray from that point be $\textbf{m}(x)$. 

The generalized Snell's law of refraction says that
\begin{equation}\label{eq:GSL}
n_1{\bf e}(x) -  \,n_2{\bf m}(x)=\lambda \,\nu(x)+\nabla \Phi(x,f(x)),
\end{equation}
for some $\lambda \in \R$; where $\nu(x)$ is the unit normal vector to surface $\Gamma$ at the incident point $(x,f(x))$. 
 
Similarly, there is also a generalized Snell's law of {\it reflection}. In fact, since both the incoming and reflected rays in this case travel in medium $I$, we have that 
\begin{equation}\label{eq:GSLreflection}
n_1\,{\bf e}(x) - n_1 \,{\bf r}(x)=\lambda \,\nu(x)+\nabla \Phi(x,f(x)),
\end{equation}
for some $\lambda\in \R$.
For a proof of these equations we refer to \cite{gps} and \cite{gutierrez-sabra:chromaticaberrationinmetasurfaces}. 

\subsection{Design of reflecting-refracting metalenses}\label{subsec:design of reflecting-refracting metasurfaces}
Let $\Omega_0,\Omega_1,\Omega_2\subset \R^2$ be compact domains with $\Omega_1,\Omega_2$ convex, and Borel measures $\rho_i$ in $\Omega_i$, $i = 0,1,2$, satisfying 
\begin{equation}\label{eq:conservation of energy multi}
\int_{\Omega_0}\rho_0(x)\,dx=\int_{\Omega_1}\rho_1(x)\,dx=\int_{\Omega_2}\rho_2(x)\,dx.
\end{equation}
In $\R^3$, we have the planes $z=0$, $z=\beta$, and a smooth surface $S_1$ with equation $z=f(x)$, at a positive distance from both planes and lying in between both planes. 
A material with refractive index $n_1$ fills the region between the plane $z=0$ and the surface $S_1$, and a material with refractive index $n_2$ fills the region between the surfaces $S_1$ and $z=\beta$. 
The problem that we consider and solve is the following.
\begin{figure}[t]
    \centering
    \includegraphics[scale=0.45]{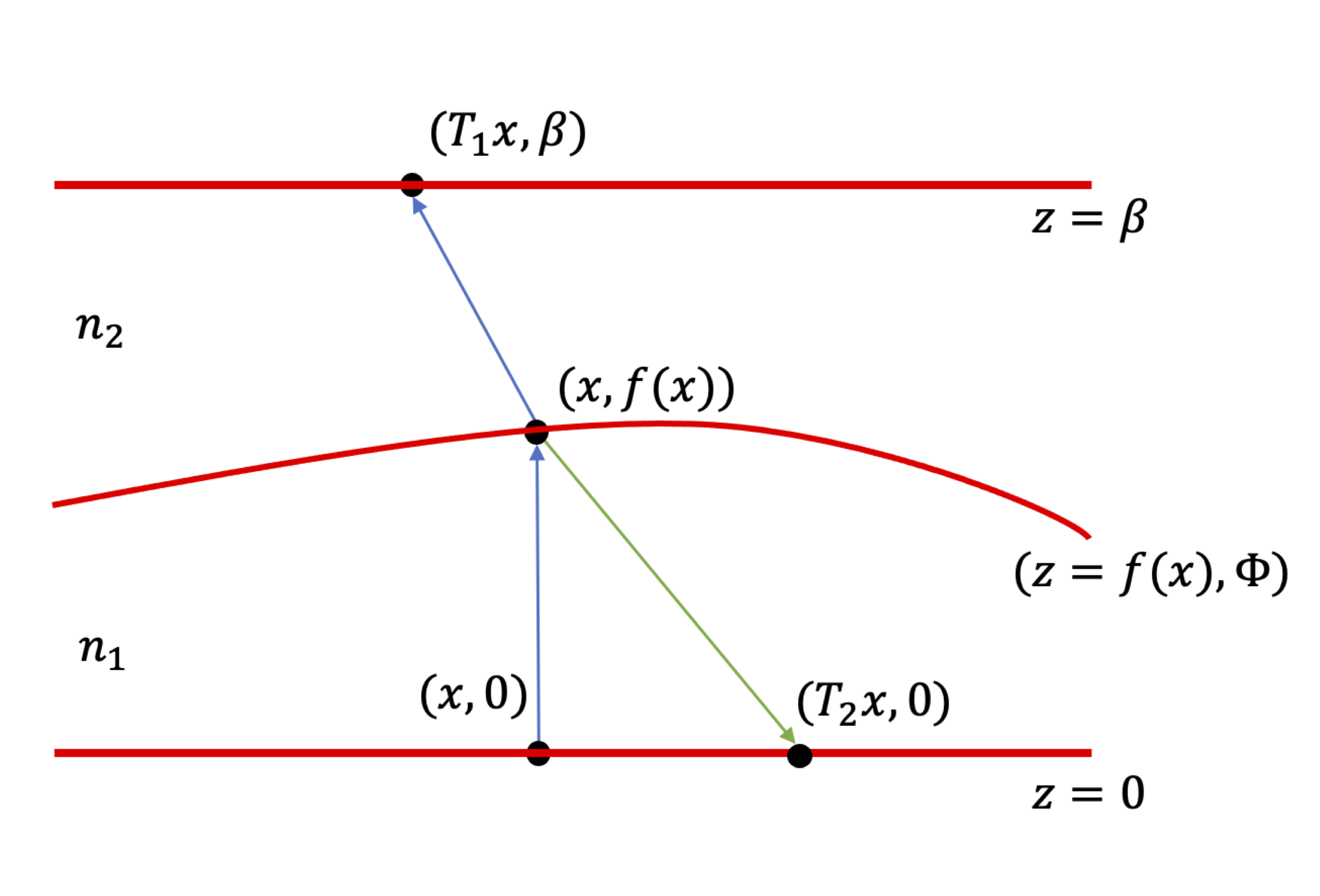}
    \caption{Refracting-reflecting metalens}
    \label{fig:refraction-reflaction}
\end{figure}

{\it We seek a phase function $\Phi$ on the surface $z=f(x)$ doing the following.}  
A ray emitted from $x=(x_1,x_2)\in \Omega_0$ with direction $\textbf{e}(x)=(0,0,1)$ strikes the surface $z=f(x)$ at a point $\(x,f(x)\)$. 
Since on $z=f(x)$ a phase $\Phi$ is defined, the ray is refracted into a unit direction ${\bf m}(x)=(m_1(x),m_2(x),m_3(x))$, $m_3(x)>0$ in accordance with the generalized Snell law (GSL) of refraction \eqref{eq:GSL} depending on the phase $\Phi$, and reaching a point $(T_1x,\beta)$ on the plane $z=\beta$ with $Tx\in \Omega_1$.
Each $\Phi$ then gives rise to a mapping $T_1:\Omega_0\to \Omega_1$.
On the other hand, the ray emitted from $x\in \Omega_0$ when it strikes $S_1$ is also reflected back into a unit direction ${\bf r}(x) = \(r_1(x),r_2(x),r_3(x)\)$, $r_3(x)<0$ according to the generalized law of reflection \eqref{eq:GSLreflection} depending on the phase $\Phi$, and reaching a point $(T_2x,0)$ on $z=0$ with $T_2x\in \Omega_2$;
see Figure \ref{fig:refraction-reflaction}.
The question is then to find a phase $\Phi$ on the surface $z=f(x)$, with $\Phi$ tangential to the surface $S_1$, i.e., $\nabla \Phi\(x,f(x)\)\cdot (-\nabla f(x),1)=0$, such that the following energy conservation conditions are met
\begin{equation}
\int_{T_1^{-1}(E)}\rho_0(x)\,dx=\int_{E}\rho_1(x)\,dx \quad E\subset \Omega_1
\label{eq:conservation of energy multi 1}
\end{equation}
\begin{equation}
\int_{T_2^{-1}(E)}\rho_0(x)\,dx=\int_{E}\rho_2(x)\,dx \quad E\subset \Omega_2.
\label{eq:conservation of energy multi 2}
\end{equation}

 Let us first analyze the trajectory of the ray.
Our light ray starts with unit direction $\textbf{e}(x) = (0, 0, 1)$ from the point $(x,0)$ to $\(x, f(x)\)\in S_1$ and from this point it is refracted to the point $(T_1x, \beta)$ with unit direction ${\bf m}(x)$. The same point $\(x, f(x)\)\in S_1$ is also reflected to the point $(T_2x, 0)$ with unit direction ${\bf r}(x)$.  Then we have that the refracted unit direction is 
\begin{equation}\label{eq:formula for m general field non planar multi}
{\bf m}(x)=\dfrac{\(T_1x-x,\beta-f(x)\)}{\sqrt{|T_1x-x|^2+\(\beta-f(x)\)^2}},
\end{equation}
and the reflected unit direction is 
\begin{equation}\label{eq:formula for r general field non planar multi}
{\bf r}(x)=\dfrac{\(T_2x-x,-f(x)\)}{\sqrt{|T_2x-x|^2+\(-f(x)\)^2}}.
\end{equation}


Applying the GSL \eqref{eq:GSL}, and \eqref{eq:GSLreflection} on the surface $z=f(x)$ yields
\begin{equation*}
    n_1\textbf{e}(x) - n_2\textbf{m}(x) = \lambda\,\nu(x,f(x)) + \nabla\Phi(x,f(x)) 
\end{equation*}
for refraction and
\begin{equation*}
    n_1\textbf{e}(x) - n_1{\bf r}(x) = \lambda'\,\nu(x,f(x)) + \nabla\Phi(x,f(x)) 
\end{equation*}
for reflection, 
where $\textbf{e} = (0,0,1)$, $\textbf{m}=(m_1,m_2,m_3)$ is the refracted unit vector, ${\bf r}=(r_1,r_2,r_3)$ is the reflected unit vector, and $\nu(x,f(x)) = (-\nabla f(x),1)$ is the normal vector to the surface $z=f(x)$ at the point $\(x,f(x)\)$, for some $\lambda,\lambda' \in \R$. Since $\Phi$ is required to be tangential, 
\begin{equation}\label{eq:tangential condition4 multi}
-\Phi_{x_1}(x,f(x))\,f_{x_1}(x)-\Phi_{x_2}(x,f(x))\,f_{x_2}(x)+\Phi_{x_3}(x,f(x))=0.
\end{equation}
We then have the equations
\begin{equation}\label{eq:general e m nonplanar multi refraction}
\begin{cases}
-n_2\,m_1(x)&=-\lambda\,f_{x_1}(x)+\Phi_{x_1}(x,f(x))\\
-n_2\,m_2(x)&=-\lambda\,f_{x_2}(x)+\Phi_{x_2}(x,f(x))\\
n_1-n_2\,m_3(x)&=\lambda+\Phi_{x_3}(x,f(x)),
\end{cases}
\end{equation}
and for reflection
\begin{equation}\label{eq:general e m nonplanar multi reflection}
\begin{cases}
-n_1\,r_1(x)&=-\lambda'\,f_{x_1}(x)+\Phi_{x_1}(x,f(x))\\
-n_1\,r_2(x)&=-\lambda'\,f_{x_2}(x)+\Phi_{x_2}(x,f(x))\\
n_1-n_1\,r_3(x)&=\lambda'+\Phi_{x_3}(x,f(x)).
\end{cases}\end{equation}
Comparing \eqref{eq:formula for m general field non planar multi} and \eqref{eq:general e m nonplanar multi refraction} yields
\begin{align}
\label{eq:formula for Tx-phix}
T_1x-x&=\sqrt{(\beta-f(x))^2+|T_1x-x|^2}\(\dfrac{1}{n_2}\,
\(\lambda\,\nabla f(x)-\(\Phi_{x_1}\(x,f(x)\),\Phi_{x_2}\(x,f(x)\)\)\)\).
\end{align}
Combining the last equation in \eqref{eq:general e m nonplanar multi refraction} with \eqref{eq:formula for m general field non planar multi} and \eqref{eq:tangential condition4 multi} yields 
\begin{align*}
\lambda&=n_1-n_2\,\dfrac{\beta-f(x)}{\sqrt{(\beta-f(x))^2+|T_1x-x|^2}}-\Phi_{x_3}\\
&=n_1-n_2\,\dfrac{\beta-f(x)}{\sqrt{(\beta-f(x))^2+|T_1x-x|^2}}-\Phi_{x_1}\,f_{x_1}-\Phi_{x_2}\,f_{x_2}.
\end{align*}
Hence
\[
\lambda\,f_{x_i}-\Phi_{x_i}=
\(n_1-n_2\,\dfrac{\beta-f(x)}{\sqrt{(\beta-f(x))^2+|T_1x-x|^2}}-\Phi_{x_1}\,f_{x_1}-\Phi_{x_2}\,f_{x_2}\)\,f_{x_i}-\Phi_{x_i},\]
for $i=1,2$, and so 
\begin{align*}
&\lambda\,\nabla f-\(\Phi_{x_1},\Phi_{x_2}\)\\
&=
\(n_1-n_2\,\dfrac{\beta-f(x)}{\sqrt{(\beta-f(x))^2+|T_1x-x|^2}}\)\nabla f
-
\(\nabla f\otimes \nabla f\)\(\Phi_{x_1},\Phi_{x_2}\)-\(\Phi_{x_1},\Phi_{x_2}\)\\
&=
n_1\,\nabla f
-n_2\,\dfrac{\beta-f(x)}{\sqrt{(\beta-f(x))^2+|T_1x-x|^2}}\,\nabla f
-
\(Id+\nabla f\otimes \nabla f\)\(\Phi_{x_1},\Phi_{x_2}\).
\end{align*}
Therefore
\begin{equation}\label{eq:main formula for T1 non flat}
\dfrac{T_1x-x+(\beta-f(x))\nabla f(x)}{\sqrt{(\beta-f(x))^2+|T_1x-x|^2}}
=\dfrac{-1}{n_2}\(Id+\nabla f\otimes \nabla f\)\(\Phi_{x_1},\Phi_{x_2}\)
+\dfrac{n_1}{n_2}\nabla f.
\end{equation}

On the other hand, from $\eqref{eq:formula for r general field non planar multi}$ we have
\[
r_3(x)=\dfrac{-f(x)}{\sqrt{|T_2x-x|^2+\(-f(x)\)^2}},\quad (r_1(x),r_2(x))=\dfrac{T_2x-x}{\sqrt{|T_2x-x|^2+\(-f(x)\)^2}}.\]
From the third equation in \eqref{eq:general e m nonplanar multi reflection} and \eqref{eq:tangential condition4 multi} we get
\[
\lambda'=n_1\(1+\dfrac{f(x)}{\sqrt{|T_2x-x|^2+\(-f(x)\)^2}}\)-\Phi_{x_1}\,f_{x_1}-\Phi_{x_2}\,f_{x_2}\]
and then from the first two equations in \eqref{eq:general e m nonplanar multi reflection}
\begin{align*}
r_i(x)&=\dfrac{1}{n_1}\(\lambda'\,f_{x_i}-\Phi_{x_i}\)\\
&=\dfrac{1}{n_1}\(\(n_1\(1+\dfrac{f(x)}{\sqrt{|T_2x-x|^2+\(-f(x)\)^2}}\)-\Phi_{x_1}\,f_{x_1}-\Phi_{x_2}\,f_{x_2}\)\,f_{x_i}-\Phi_{x_i}\)\\
&=
\(1+\dfrac{f(x)}{\sqrt{|T_2x-x|^2+\(-f(x)\)^2}}\)f_{x_i}-\dfrac{1}{n_1}\Phi_{x_1}\,f_{x_1}f_{x_i}-\dfrac{1}{n_1}\Phi_{x_2}\,f_{x_2}\,f_{x_i}-\dfrac{1}{n_1}\Phi_{x_i},
\end{align*}
for $i=1,2$.
Hence
\[
\(r_1(x),r_2(x)\)
=
\(1+\dfrac{f(x)}{\sqrt{|T_2x-x|^2+\(-f(x)\)^2}}\)\nabla f -\dfrac{1}{n_1}\(Id+\nabla f\otimes \nabla f\)
\(\Phi_{x_1},\Phi_{x_2}\).
\]
Therefore
\[
\dfrac{T_2x-x}{\sqrt{|T_2x-x|^2+\(-f(x)\)^2}}
=
\(1+\dfrac{f(x)}{\sqrt{|T_2x-x|^2+\(-f(x)\)^2}}\)\nabla f -\dfrac{1}{n_1}\(Id+\nabla f\otimes \nabla f\)
\(\Phi_{x_1},\Phi_{x_2}\).
\]
But substituting $\dfrac{1}{n_1}\(Id+\nabla f\otimes \nabla f\)
\(\Phi_{x_1},\Phi_{x_2}\)$ from its value in \eqref{eq:main formula for T1 non flat} yields
\[
\dfrac{T_2x-x}{\sqrt{|T_2x-x|^2+\(-f(x)\)^2}}
=
\dfrac{n_2}{n_1}
\dfrac{T_1x-x+(\beta-f(x))\nabla f(x)}{\sqrt{(\beta-f(x))^2+|T_1x-x|^2}}
+
\(\dfrac{f(x)}{\sqrt{|T_2x-x|^2+\(-f(x)\)^2}}\)\nabla f,\]
which after rearranging the terms, it is 
\begin{equation}\label{eq:relation between T1 and T2 and f}
\dfrac{T_2x-x-f(x)\nabla f}{\sqrt{|T_2x-x|^2+\(-f(x)\)^2}}
=
\dfrac{n_2}{n_1}
\dfrac{T_1x-x+(\beta-f(x))\nabla f(x)}{\sqrt{(\beta-f(x))^2+|T_1x-x|^2}}
.
\end{equation}

We now connect this with multi-marginal optimal transport theory. Let us introduce the cost function 
\begin{equation}\label{multi-marginal cost}
c(x_1,x_2,x_3)=c_1(x_1,x_2)+c_2(x_1,x_3),
\end{equation} where $(x_1,x_2,x_3)\in \Omega_0\times\Omega_1\times\Omega_2$; $c_1(x_1,x_2)=n_2\sqrt{|x_1-x_2|^2+(\beta-f(x_1))^2}$, and $c_2(x_1,x_3)=n_1\sqrt{|x_1-x_3|^2+f(x_1)^2}$.

Note that 
\[
\nabla_{x_1}\(c_1(x_1,x_2)\)= n_2\dfrac{x_1-x_2-(\beta-f(x_1))\nabla f}{\sqrt{|x_1-x_2|^2+(\beta-f(x_1))^2}},\quad
\nabla_{x_1}\(c_2(x_1,x_3)\)=n_1\dfrac{x_1-x_3+f(x_1)\nabla f}{\sqrt{|x_3-x_1|^2+f(x_1)^2}},
\]
and so
\[
\nabla_{x_1} c_1(x_1,T_1x_1)=- n_2\dfrac{T_1x_1-x_1+(\beta-f(x_1))\nabla f(x_1)}{\sqrt{(\beta-f(x_1))^2+|T_1x_1-x_1|^2}},\]
\[
\nabla_{x_1} c_2(x_1,T_2x_1)=- n_1\dfrac{T_2x_1-x_1-f(x_1)\nabla f}{\sqrt{|T_2x_1-x_1|^2+\(-f(x_1)\)^2}}.
\]
Combining these with $\eqref{eq:relation between T1 and T2 and f}$, we get that $\nabla_{x_1}c_1(x_1,T_1x_1) = \nabla_{x_1}c_2(x_1,T_2x_1)$. 
Since $c(x_1,x_2,x_3)=c_1(x_1,x_2)+c_2(x_1,x_3)$, we have 
\begin{align*}
\nabla_{x_1} c(x_1,T_1x_1,T_2x_1) &= \nabla_{x_1}c_1(x_1,T_1x_1) + \nabla_{x_1}c_2(x_1,T_2x_1) \\
&= 2\nabla_{x_1}c_1(x,T_1x)= 2\nabla_{x_1}c_2(x,T_2x).
\end{align*}
Then, from $\eqref{eq:main formula for T1 non flat}$ and $\eqref{eq:relation between T1 and T2 and f}$, 
\begin{equation}\label{eq: equation for Phi multi}
\dfrac{1}{2}\nabla_{x_1}c(x_1,T_1x_1,T_2x_1) = \(Id+\nabla f\otimes \nabla f\)\(\Phi_{x_1},\Phi_{x_2}\) - n_1\nabla f,
\end{equation}
showing that the pair of maps $(T_1, T_2)$ and $\Phi$ are related via the cost $c$. In fact, if the pair $(T_1,T_2)$ is the optimal pair for the cost $c$, then it satisfies the energy conservation conditions \eqref{eq:conservation of energy multi 1} and \eqref{eq:conservation of energy multi 2}, and the phase $\Phi$ is determined by \eqref{eq: equation for Phi multi}.

The results of this section are summarized in the following Theorem.

\begin{theorem}\label{thm:main result representation of the phase}
Suppose the surface $S_1$ is the graph of a function $f$; and let $(T_1,T_2)$ be the optimal pair of maps for the cost \eqref{multi-marginal cost} with densities $\rho_0,\rho_1$ and $\rho_2$ satisfying \eqref{eq:conservation of energy multi} with $\rho_0$ strictly positive. 
If the map $(x_2,x_3)\mapsto \nabla_{x_1}c(x_1,x_2,x_3)$ is injective for all points $x_1\in \Omega_0$, then the desired phase $\Phi$ solving the reflection-refraction problem described in this section satisfies
\[
\(\Phi_{x_1},\Phi_{x_2}\)
=
\(Id - \dfrac{\nabla f\otimes \nabla f}{1 + |\nabla f|^2}\)\(\dfrac{1}{2}Df_1(x) + n_1 \nabla f(x)\),
\] where $f_1$ is the function in Theorem \ref{thm:representation formula for multimarginal}, and $\Phi_{x_i}$, $i=1,2$ is evaluated at $(x,f(x))$.
\end{theorem}

\subsection{Analysis of the cost \eqref{multi-marginal cost}}\label{analysisofmultimarginalcost}

In this section, we identify conditions on the function $f$ describing the surface $S_1$ under which the map $(x_2, x_3) \mapsto \nabla_{x_1}c(x_1,x_2,x_3)$ is injective for each $x_1\in \Omega_0$, where $c$ denotes the cost from \eqref{multi-marginal cost}. Under these conditions, Theorem \ref{thm:representation formula for multimarginal} yields the representation formula
\[ 
\nabla_{x_1}f_1(x_1) = \nabla_{x_1}c(x_1,T_1x_1, T_2x_1)
\]
for almost every $x_1 \in \Omega_0$, for some function $f_1$, which, in turn, yields the conclusion of Theorem \ref{thm:main result representation of the phase}.

\begin{proposition}
Let $\Omega_j\subset\R^2$, $0\leq j\leq 2$, be compact domains with $\Omega_1,\Omega_2$ convex, and let $c$ be the cost given by \eqref{multi-marginal cost} with $f\in C^1(\Omega_0)$, $0<f(x)<\beta$ for all $x\in \Omega_0$.
Let $N_0=\max_{z_1\in \Omega_0,z_2\in \Omega_1}|z_1-z_2|$, $N_1=\max_{x_1\in \Omega_0,z\in \Omega_2}|x_1-z|$, $m_0=\min_{x_1\in \Omega_0}\(\beta -f(x_1)\)$, and $m_1=\min_{x_1\in\Omega_0} f(x_1)$.
There exists a positive constant $C_0$, depending on $N_0,N_1,m_0, m_1$ and $\beta$, such that if $\|\nabla f\|_\infty<C_0$, then the map $(x_2, x_3) \mapsto \nabla_{x_1}c(x_1,x_2,x_3)$, from $\Omega_1\times \Omega_2$ to $\R^2$, is injective for each $x_1\in \Omega_0$.
\end{proposition}

\begin{proof}
We begin with a general framework. Let $c(x_1,x_2,\ldots,x_N)$ be a $C^1$ cost function with variables $x_j \in \mathbb{R}^n$, and define $X = (x_2, \ldots, x_N)$ so that $c = c(x_1, X)$. We consider the gradient of $c$ with respect to $x_1$:
\[ 
\nabla_{x_1}c(x_1,X) = \left(\frac{\partial c}{\partial x_1^1}(x_1,X),\ldots,\frac{\partial c}{\partial x_1^n}(x_1,X)\right),
\]
where $x_1 = (x_1^1,\ldots,x_1^n)$. Suppose we are given that $\nabla_{x_1}c(x_1,X) = \nabla_{x_1}c(x_1,Y)$. Our goal is to show that $X = Y$ under suitable conditions on $c$.

%
%

Define the path $X(s) = sX + (1-s)Y$, and note that
\begin{align*}
0 &= \frac{\partial c}{\partial x_1^j}(x_1,X) - \frac{\partial c}{\partial x_1^j}(x_1,Y) 
= \int_0^1 \nabla_{x_2,\ldots,x_N} \left( \frac{\partial c}{\partial x_1^j} \right)(x_1,X(s)) \cdot \dot{X}(s)\,ds \\
&= \int_0^1 \nabla_{x_2,\ldots,x_N} \left( \frac{\partial c}{\partial x_1^j} \right)(x_1,X(s)) \cdot (X - Y)\,ds.
\end{align*}

%

Set 
\[
m_j(s)=\nabla_{x_2,\cdots ,x_N}\(\dfrac{\partial c}{\partial x_1^j}\)(x_1,X(s)),
\] 
$1\leq j\leq n$,
that is, $m_j$ is a vector with $n(N-1)$ components.
Let $M(s)$ be the matrix whose rows are the vectors $m_j(s)$; $M(s)$ is then a matrix having $n$ rows and $n(N-1)$ columns. So we can write
\[
0=\nabla_{x_1}c(x_1,X)-\nabla_{x_1}c(x_1,Y)
= 
\int_0^1 M(s) (X-Y)^t\,ds.
\]
To conclude injectivity, let us take any $n(N-1)\times n$ matrix $A$ and since $X-Y$ has $n(N-1)$ components 
it follows that $(X-Y)A$ is a vector with $n$ components. So dotting the last integral identity with $(X-Y)A$
we get 
\[
0=(X-Y)A\cdot \(\nabla_{x_1}c(x_1,X)-\nabla_{x_1}c(x_1,Y)\)
= 
\int_0^1 (X-Y)A M(s) (X-Y)^t\,ds
\]
for any $n(N-1)\times n$ matrix $A$. 
Our strategy is to choose a matrix $A$ such that $AM(s)$ is positive definite for all $s \in [0,1]$, which then implies $X = Y$.


We next apply the framework above with $n=2$ and $N=3$, and proceed to analyze the mixed second-order derivatives of $c$ with respect to $x_1$ and $(x_2,x_3)$. 
The remainder of the proof involves detailed expressions for the gradients, second derivatives, and matrix formulations. 

We have 
\begin{align*}
\nabla_{x_1}c(x_1,x_2,x_3)
&=\nabla_{x_1}c_1(x_1,x_2)+\nabla_{x_1}c_2(x_1,x_3)\\
&=\(\dfrac{\partial c_1}{\partial x_1^1}(x_1,x_2),\dfrac{\partial c_1}{\partial x_1^2}(x_1,x_2) \)
+
\(\dfrac{\partial c_2}{\partial x_1^1}(x_1,x_3),\dfrac{\partial c_2}{\partial x_1^2}(x_1,x_3) \)\\
&=
\(\dfrac{\partial c_1}{\partial x_1^1}(x_1,x_2)+\dfrac{\partial c_2}{\partial x_1^1}(x_1,x_3),
\dfrac{\partial c_1}{\partial x_1^2}(x_1,x_2)+\dfrac{\partial c_2}{\partial x_1^2}(x_1,x_3)\)
\end{align*}
and for $j=1,2$
\begin{align*}
&\nabla_{x_2,x_3}\(\dfrac{\partial c}{\partial x_1^j}\)
=\nabla_{x_2,x_3} \(\dfrac{\partial c_1}{\partial x_1^j}(x_1,x_2)+\dfrac{\partial c_2}{\partial x_1^j}(x_1,x_3)\)\\
&=\resizebox{0.97\textwidth}{!}{$\(\dfrac{\partial}{\partial x_2^1}\(\dfrac{\partial c_1}{\partial x_1^j}(x_1,x_2)+\dfrac{\partial c_2}{\partial x_1^j}(x_1,x_3)\), \dfrac{\partial}{\partial x_2^2}\(\dfrac{\partial c_1}{\partial x_1^j}(x_1,x_2)+\dfrac{\partial c_2}{\partial x_1^j}(x_1,x_3)\),
\dfrac{\partial}{\partial x_3^1}\(\dfrac{\partial c_1}{\partial x_1^j}(x_1,x_2)+\dfrac{\partial c_2}{\partial x_1^j}(x_1,x_3)\),
\dfrac{\partial}{\partial x_3^2}\(\dfrac{\partial c_1}{\partial x_1^j}(x_1,x_2)+\dfrac{\partial c_2}{\partial x_1^j}(x_1,x_3)\)\)$}\\
&=
\(
\dfrac{\partial^2 c_1}{\partial x_2^1\partial x_1^j}(x_1,x_2)
, \dfrac{\partial^2 c_1}{\partial x_2^2 \partial x_1^j}(x_1,x_2),
\dfrac{\partial^2c_2}{\partial x_3^1 \partial x_1^j}(x_1,x_3),
\dfrac{\partial^2 c_2}{\partial x_3^2\partial x_1^j} (x_1,x_3)\).
\end{align*}
Therefore
 \[
\resizebox{1\textwidth}{!}{$m_1(s)=\(
\dfrac{\partial^2 c_1}{\partial x_2^1\partial x_1^1}(x_1,sx_2+(1-s)y_2)
, \dfrac{\partial^2c_1}{\partial x_2^2 \partial x_1^1}(x_1,sx_2+(1-s)y_2),
\dfrac{\partial^2c_2}{\partial x_3^1 \partial x_1^1}(x_1,sx_3+(1-s)y_3),
\dfrac{\partial^2c_2}{\partial x_3^2\partial x_1^1} (x_1,sx_3+(1-s)y_3)\)$}
\]

 \[
\resizebox{1\textwidth}{!}{$m_2(s)=\(
\dfrac{\partial^2 c_1}{\partial x_2^1\partial x_1^2}(x_1,sx_2+(1-s)y_2)
, \dfrac{\partial^2c_1}{\partial x_2^2 \partial x_1^2}(x_1,sx_2+(1-s)y_2),
\dfrac{\partial^2c_2}{\partial x_3^1 \partial x_1^2}(x_1,sx_3+(1-s)y_3),
\dfrac{\partial^2c_2}{\partial x_3^2\partial x_1^2} (x_1,sx_3+(1-s)y_3)\)$}.
\]
%
Therefore setting
\[
M_1(s)
=
\begin{bmatrix}
\dfrac{\partial^2 c_1}{\partial x_2^1\partial x_1^1}(x_1,sx_2+(1-s)y_2)
& \dfrac{\partial^2c_1}{\partial x_2^2 \partial x_1^1}(x_1,sx_2+(1-s)y_2)
\\
\dfrac{\partial^2 c_1}{\partial x_2^1\partial x_1^2}(x_1,sx_2+(1-s)y_2)
& \dfrac{\partial^2c_1}{\partial x_2^2 \partial x_1^2}(x_1,sx_2+(1-s)y_2)
\end{bmatrix}
\]
and
\[
M_2(s)
=
\begin{bmatrix}
\dfrac{\partial^2c_2}{\partial x_3^1 \partial x_1^1}(x_1,sx_3+(1-s)y_3)
&\dfrac{\partial^2c_2}{\partial x_3^2\partial x_1^1} (x_1,sx_3+(1-s)y_3)\\
\dfrac{\partial^2c_2}{\partial x_3^1 \partial x_1^2}(x_1,sx_3+(1-s)y_3)
& \dfrac{\partial^2c_2}{\partial x_3^2\partial x_1^2} (x_1,sx_3+(1-s)y_3)
\end{bmatrix}\]
we can write
\[
M(s)
=
\begin{bmatrix}
M_1(s) & M_2(s)
\end{bmatrix}.
\]

We have from \cite[Formula (4.12)]{2025-altinergutierrez:metasurfaces}
\begin{align}\label{eq:hessian of c beta}
\dfrac{\partial^2 c}{\partial x\partial y}
&=
\dfrac{-1}{c(x,y)}
\(\(\dfrac{x-y}{c(x,y)} \)\otimes \(\dfrac{(\beta-f(x))\nabla f+y-x}{c(x,y)} \)+Id \)\notag\\
&= -\dfrac{1}{c(x,y)^3} \(\(x-y\)\otimes \((\beta-f(x))\nabla f+y-x \)+c(x,y)^2\,Id \),
\end{align}
for $c$ of the form $c(x,y)=\sqrt{(f(x)-\beta)^2+|x-y|^2}$.

From \eqref{eq:hessian of c beta} with $g(y)=\beta$ the matrix $M_1$ reads
\begin{align*}
M_1(s)
&=
\dfrac{-1}{c_1(x_1,sx_2+(1-s)y_2)^3}\\
&\qquad \(\(x_1-y_2 -s(x_2-y_2) \)\otimes \((\beta-f(x_1))\nabla f(x_1)+s(x_2-y_2)+y_2-x_1 \)+c_1(x_1,sx_2+(1-s)y_2)^2 Id \)
\end{align*}
and \eqref{eq:hessian of c beta} with $\beta=0$ yields
\begin{align*} M_2(s)
&=
\dfrac{-1}{c_2(x_1,sx_3+(1-s)y_3)^3}\\
&\qquad 
\(\(x_1-y_3-s(x_3-y_3) \)
\otimes \(-f(x_1)\nabla f(x_1)+s(x_3-y_3)+y_3-x_1 \)+c_2(x_1,sx_3+(1-s)y_3)^2 Id \).
\end{align*}
If 
$
A=
\begin{bmatrix}
A_1 \\
A_2
\end{bmatrix}
$
with $A_1,A_2$ $2\times 2$ matrices, we have
\[
A M(s)
=
\begin{bmatrix}
A_1\\ 
A_2
\end{bmatrix}
\begin{bmatrix}
M_1(s) & M_2(s)
\end{bmatrix}
=
\begin{bmatrix}
A_1 M_1(s) & A_1 M_2(s)\\
A_2 M_1(s) & A_2 M_2(s)
\end{bmatrix}.
\]
Let $A$ be the $4\times 2$ matrix of the form
\[
\begin{bmatrix}
-I_2\\
0
\end{bmatrix}
\]
with $I_2$ the identity $2\times 2$ matrix, so
\[
A M(s)
=
\begin{bmatrix}
-I_2\\ 
0
\end{bmatrix}
\begin{bmatrix}
M_1(s) & M_2(s)
\end{bmatrix}
=
\begin{bmatrix}
-M_1(s) & -M_2(s)\\
0 & 0
\end{bmatrix}.
\]
Hence
\begin{align}\label{eq:integral representation with arbitrary matrix A}
0&=(X-Y)A\cdot \(\nabla_{x_1}c(x_1,X)-\nabla_{x_1}c(x_1,Y)\)
= 
\int_0^1 (X-Y)A M(s) (X-Y)^t\,ds\notag\\
&=
\int_0^1 (X-Y)\begin{bmatrix}
-M_1(s) & -M_2(s)\\
0 & 0
\end{bmatrix} (X-Y)^t\,ds.
\end{align}
Notice that if $X=(x_2^1,x_2^2,x_3^1,x_3^2)$, $Y=(y_2^1,y_2^2,y_3^1,y_3^2)$, then
\begin{align}\label{eq:formula for x-y M1 M2 x-y^t}
&(X-Y)\begin{bmatrix}
-M_1(s) & -M_2(s)\\
0 & 0
\end{bmatrix} (X-Y)^t\notag\\
&=
[x_2^1-y_2^1,x_2^2-y_2^2,x_3^1-y_3^1,x_3^2-y_3^2]
\begin{bmatrix}
-M_1(s) & -M_2(s)\\
0 & 0
\end{bmatrix}
\begin{bmatrix}
x_2^1-y_2^1\\
x_2^2-y_2^2\\
x_3^1-y_3^1\\
x_3^2-y_3^2
\end{bmatrix}\notag\\
&=
[x_2^1-y_2^1,x_2^2-y_2^2]
\begin{bmatrix}
-M_1(s)
\end{bmatrix}\begin{bmatrix}
x_2^1-y_2^1\\
x_2^2-y_2^2
\end{bmatrix}
+
[x_2^1-y_2^1,x_2^2-y_2^2]
\begin{bmatrix}
-M_2(s)
\end{bmatrix}\begin{bmatrix}
x_2^1-y_2^1\\
x_2^2-y_2^2
\end{bmatrix}.
\end{align}
Assuming conditions on $f$, we shall prove that the matrices $-M_1(s)$ and $-M_2(s)$ are both positive definite.

Let us prove first that $-M_1(s)$ is positive definite.
To do this, we show that the entry (1,1) of the matrix $-M_1(s)$ and $\det -M_1(s)>0$ are both positive.
This entry is equal to
\begin{align*}
&m_{11}\\
&=\dfrac{1}{c_1(x_1,sx_2+(1-s)y_2)^3}\\
&\(c_1(x_1,sx_2+(1-s)y_2)^2-\(x_1^1-y_2^1 -s(x_2^1-y_2^1)\)^2
+\(x_1^1-y_2^1 -s(x_2^1-y_2^1) \) \((\beta-f(x_1))\partial_{x_1^1} f(x_1)\)\)\\
&=\dfrac{1}{c_1(x_1,sx_2+(1-s)y_2)^3}\\
& \(n_2^2\,\((\beta -f(x_1))^2+|x_1-(sx_2+(1-s)y_2)|^2\)
-
\(x_1^1-y_2^1 -s(x_2^1-y_2^1)\)^2\right. \\
&\qquad \qquad \qquad \qquad \qquad \left. +\(x_1^1-y_2^1 -s(x_2^1-y_2^1) \) \((\beta-f(x_1))\partial_{x_1^1} f(x_1)\) \)
.
\end{align*} 
So $m_{11}$ is positive if we show a positive lower bound for 
\begin{align*}
L&:=\,n_2^2\,\((\beta -f(x_1))^2+|x_1-(sx_2+(1-s)y_2)|^2\)\\
&\qquad -\(x_1^1-y_2^1 -s(x_2^1-y_2^1)\)^2 +\(x_1^1-y_2^1 -s(x_2^1-y_2^1) \) \((\beta-f(x_1))\partial_{x_1^1} f(x_1)\) \\
&\geq 
n_2^2\,\((\beta -f(x_1))^2+|x_1-(sx_2+(1-s)y_2)|^2\)\\
&\qquad -|x_1-(sx_2+(1-s)y_2)|^2 +\(x_1^1-y_2^1 -s(x_2^1-y_2^1) \) \((\beta-f(x_1))\partial_{x_1^1} f(x_1)\).
\end{align*}
Since $\Omega_1$ is a convex domain, $sx_2+(1-s)y_2\in \Omega_1$ for $0\leq s\leq 1$ and $x_2,y_2\in \Omega_1$, and for each $x_1\in \Omega_0$
\begin{equation}\label{eq:definition of N0}
|x_1-y_2- s(x_2-y_2)|\leq \max_{z_1\in \Omega_0,z_2\in \Omega_1}|z_1-z_2|:=N_0.
\end{equation}
Hence
\begin{align*}
&\left| \(x_1^1-y_2^1 -s(x_2^1-y_2^1) \) \((\beta-f(x_1))\partial_{x_1^1} f(x_1)\)\right|\\
&\qquad \leq |x_1-(sx_2+(1-s)y_2)| \,\left|(\beta-f(x_1))\nabla_{x_1} f(x_1)\right|\leq N_0 \,\left|(\beta-f(x_1))\nabla_{x_1} f(x_1)\right|\end{align*}
and so 
\begin{align*}
&\(x_1^1-y_2^1 -s(x_2^1-y_2^1) \) \((\beta-f(x_1))\partial_{x_1^1} f(x_1)\)
\geq -N_0 \,\left|(\beta-f(x_1))\nabla_{x_1} f(x_1)\right|.
\end{align*}
Hence
\begin{align*}
L&\geq n_2^2\,(\beta -f(x_1))^2+\(n_2^2-1\)|x_1-(sx_2+(1-s)y_2)|^2
-N_0\,\left|(\beta-f(x_1))\nabla_{x_1} f(x_1)\right|.
\end{align*}
Since refractive indices are greater than or equal to one, we can drop the term containing $|x_1-(sx_2+(1-s)y_2)|^2$ to get
\begin{align*}
L&\geq n_2^2\,(\beta -f(x_1))^2
-N_0\,\left|(\beta-f(x_1))\nabla_{x_1} f(x_1)\right| \geq 
n_2^2\,m_0^2
-N_0\,\beta\,\|\nabla f\|_\infty,
\end{align*}
where 
\begin{equation*}\label{eq:definition of m0}
\min_{x_1\in \Omega_0}\(\beta -f(x_1)\)=m_0>0.
\end{equation*} 
If 
\begin{equation}\label{eq:condition for m11 positive when n2bigger than 1}
\|\nabla f\|_\infty< \dfrac{n_2^2 \,m_0^2}{\beta\,N_0},
\end{equation}
then we obtain that $m_{11}\geq C_0>0$ for all $0\leq s\leq 1$.



We next show that $\det -M_1(s)>0$. 
In fact,
from the Sherman-Morrison determinant formula (see \cite[Formula after (4.12)]{2025-altinergutierrez:metasurfaces})
\begin{align}\label{eq:determinant formula for general c}
\det \dfrac{\partial^2 c}{\partial x\partial y}
&=\(\dfrac{-1}{c(x,y)}\)^n\,\(1+\dfrac{x-y}{c(x,y)}\cdot \dfrac{(\beta-f(x))\nabla f+y-x}{c(x,y)}\)\notag\\
&=\(\dfrac{-1}{c(x,y)}\)^n\,\dfrac{1}{c(x,y)^2}\(c(x,y)^2+
(\beta-f(x))\,(x-y)\cdot \nabla f(x)-|x-y|^2\)\notag\\
&=
\(\dfrac{-1}{c(x,y)}\)^n\,\dfrac{1}{c(x,y)^2}
\((\beta-f(x))^2+(\beta-f(x))\,(x-y)\cdot \nabla f(x)\)\notag\\
&=
\(\dfrac{-1}{c(x,y)}\)^n\,\dfrac{\beta-f(x)}{c(x,y)^2}
\(\beta-f(x)+(x-y)\cdot \nabla f(x)\).
\end{align}
Hence 
{\small \begin{align*}
&\det -M_1(s)=\det M_1(s)\\
&=
\(\dfrac{-1}{c_1(x_1,sx_2+(1-s)y_2)}\)^2\, \dfrac{(\beta-f(x_1))}{c_1(x_1,sx_2+(1-s)y_2)^2}
\(\beta-f(x_1)+(x_1-sx_2-(1-s)y_2))\cdot \nabla f(x_1)\).
\end{align*}
}
For this determinant to be positive we need
\[
\beta-f(x_1)+(x_1-sx_2-(1-s)y_2))\cdot \nabla f(x_1)>0
\]
for all $x_1\in \Omega_0$ and all $x_2,y_2\in \Omega_1$.
As before, for $z_2\in \Omega_1$, we have $|(x_1-z_2)\cdot \nabla f(x_1)|\leq N_0 \|\nabla f\|_\infty $, and hence  
\[
\beta-f(x_1)+(x_1-sx_2-(1-s)y_2))\cdot \nabla f(x_1)
\geq 
m_0-N_0 \|\nabla f\|_\infty.
\]
Therefore, if 
\begin{equation}\label{eq:condition on the gradient of f}
\|\nabla f\|_\infty<\dfrac{m_0}{N_0}, 
\end{equation}
then the determinant is positive.


Similarly, we prove under similar assumptions on the gradient of $f$ that the matrix $-M_2(s)$ for the cost $c_2$ is also positive definite. So, from \eqref{eq:integral representation with arbitrary matrix A} and \eqref{eq:formula for x-y M1 M2 x-y^t} we get that $x_2=y_2$. 

To prove that $-M_2(s)$ is positive definite, we notice that the cost $c_2$ has a similar form of the cost $c_1$ where the function $\beta-f(x_1)$ is replaced by $f(x_1)$ and the graph of $f(x_1)$ is at a positive distance from the graph of $z=0$. 
From the form of $-M_2(s)$, its entry $1,1$ equals to
\begin{align*}
&\dfrac{1}{c_2(x_1,sx_3+(1-s)y_3)^3}\\
&\(c_2(x_1,sx_3+(1-s)y_3)^2-\(x_1^1-y_3^1 -s(x_3^1-y_3^1)\)^2
+\(x_1^1-y_3^1 -s(x_3^1-y_3^1) \) \((-f(x_1))\partial_{x_1^1} f(x_1)\)\)\\
&=\dfrac{1}{c_2(x_1,sx_3+(1-s)y_3)^3}\\
& \(n_1^2\,\(f(x_1)^2+|x_1-(sx_3+(1-s)y_3)|^2\)
-
\(x_1^1-y_3^1 -s(x_3^1-y_3^1)\)^2\right. \\
&\qquad \qquad \qquad \qquad \qquad \left. -\(x_1^1-y_3^1 -s(x_3^1-y_3^1) \) \(f(x_1)\,\partial_{x_1^1} f(x_1)\) \)
:=\dfrac{1}{c_2(x_1,sx_3+(1-s)y_3)^3}\,\tilde L
.
\end{align*}  
If we let $N_1=\max_{x_1\in \Omega_0,z\in \Omega_2}|x_1-z|$, $m_1=\min_{x_1\in\Omega_0} f(x_1)>0$, since $\Omega_2$ is convex, a calculation similar to the one done before yields
\[
\tilde L\geq n_1^2 \,m_1^2 +(n_1^2-1)\,|x_1-(y_3+s(x_3-y_3))|^2 -N_1\,\beta \,\|\nabla f\|_\infty.
\] 

Since $n_1\geq 1$, dropping the middle term we then get that the entry $1,1$ of the matrix $-M_2(s)$ is strictly positive if 
\begin{equation}\label{eq:condition for m11 positive when n1bigger than 1}
\|\nabla f\|_\infty< \dfrac{n_1^2 \,m_1^2}{\beta\,N_1}.
\end{equation}
Next we show that $\det -M_2(s)$ is positive.
Applying formula \eqref{eq:determinant formula for general c} for the cost $c_2$ we have that $\beta=0$ and so
{\small \begin{align*}
&\det -M_2(s)=\det M_2(s)\\
&=
\(\dfrac{-1}{c_2(x_1,sx_3+(1-s)y_3)}\)^2\, \dfrac{(-f(x_1))}{c_2(x_1,sx_3+(1-s)y_3)^2}
\(-f(x_1)+(x_1-sx_3-(1-s)y_3))\cdot \nabla f(x_1)\)\\
&=
\(\dfrac{1}{c_2(x_1,sx_3+(1-s)y_3)}\)^4\, 
\(f(x_1)^2- f(x_1)\,(x_1-sx_3-(1-s)y_3))\cdot \nabla f(x_1)\).
\end{align*}
}
For this determinant to be positive we need
\[
f(x_1)^2- f(x_1)\,(x_1-sx_3-(1-s)y_3))\cdot \nabla f(x_1)>0
\]
for all $x_1\in \Omega_0$ and all $x_3,y_3\in \Omega_2$.
As before, for $x_1\in \Omega_0$ and $z_2\in \Omega_2$, we have $|(x_1-z_2)\cdot \nabla f(x_1)|\leq N_1 \|\nabla f\|_\infty $, and hence  
\[
f(x_1)-(x_1-sx_3-(1-s)y_3))\cdot \nabla f(x_1)
\geq 
m_1-N_1\, \|\nabla f\|_\infty.
\]

Therefore, if 
\begin{equation}\label{eq:condition on the gradient of f for M2}
\|\nabla f\|_\infty<\dfrac{m_1}{N_1}, 
\end{equation}
then the determinant is positive.

Finally, it remains to prove that $x_3=y_3$. We proceed as before but taking $A=\begin{bmatrix}
0\\
-I_2
\end{bmatrix}$ 
to get that
\[
A M(s)
=
\begin{bmatrix}
0\\
-I_2
\end{bmatrix}
\begin{bmatrix}
M_1(s) & M_2(s)
\end{bmatrix}
=
\begin{bmatrix}
0 & 0\\
-M_1(s) & -M_2(s)
\end{bmatrix},
\]
and similarly to \eqref{eq:integral representation with arbitrary matrix A} we have
\[
0=
\int_0^1 (X-Y)\begin{bmatrix}
0 & 0\\
-M_1(s) & -M_2(s)
\end{bmatrix} (X-Y)^t\,ds
\]
which from the positivity of $-M_i(s)$ yields $x_3=y_3$.

Consequently, assuming that $C_0$ is the minimum of the constants in \eqref{eq:condition for m11 positive when n2bigger than 1}, \eqref{eq:condition on the gradient of f}, \eqref{eq:condition for m11 positive when n1bigger than 1}, and \eqref{eq:condition on the gradient of f for M2}, we derive the injectivity of the map $(x_2,x_3)\mapsto \nabla_{x_1}c(x_1,x_2,x_3)$ for all points $x_1\in \Omega_0$.
\end{proof}


%
%
%
%
\bibliography{monamp_june19-IREM}
\bibliographystyle{amsalpha}

\end{document}